\documentclass[reqno]{amsart}
\usepackage{graphicx}
\usepackage{amssymb}
\usepackage{amsfonts}
\usepackage{indentfirst}
\usepackage{amssymb,amsmath,amsthm}
\usepackage{latexsym,bm}
\usepackage{graphicx}
\usepackage{times}
\usepackage{amsfonts}
\usepackage{indentfirst}

\usepackage[colorlinks,linkcolor=black,anchorcolor=blue,citecolor=blue]{hyperref}
\usepackage[dvipsnames]{pstricks}

\newtheorem{theorem}{Theorem}[section]
\newtheorem{definition}[theorem]{Definition}
\newtheorem{lemma}[theorem]{Lemma}
\newtheorem{corollary}[theorem]{Corollary}
\newtheorem{remark}[theorem]{Remark}
\newtheorem{proposition}[theorem]{Propositon}
\newtheorem{example}[theorem]{Example}
\def\r#1{\uppercase\expandafter{\romannumeral#1}}

\numberwithin{equation}{section}

\def\a{\alpha}
\def\b{\beta}
\def\Qp{\mathbb{Q}_{p}}
\def\Zp{\mathbb{Z}_{p}}

\def\Z{\mathbb{Z}}
\def\N{\mathbb{N}}
\def\P{\mathbb{P}^{1}(\mathbb{Q}_p)}

\def\Q{\mathbb{Q}}

\def\K{\mathbb K}
\def\D{\mathbb{D}}
\def\CD{\overline{\mathbb{D}}}
\def\SP{\mathbb{S}}
\def\O{{\mathcal O}_K}
\def\PK{\mathbb{P}^{1}(K)}
\def\U{\mathbb U}
\def\V{\mathbb V}

\def\a{\alpha}
\def\b{\beta}
\def\val{v_{\pi}}

\begin{document}

\title[On minimal decomposition of $p$-adic homographic dynamical systems]{On minimal decomposition of $p$-adic homographic dynamical systems }

\author{Aihua Fan}
\address{LAMFA, UMR 7352, CNRS,
Universit\'e de Picardie Jules Verne, 33, Rue Saint Leu, 80039
Amiens Cedex 1, France}
\email{ai-hua.fan@u-picardie.fr}

\author{Shilei Fan}
\address{Institute of Applied Mathematics, AMSS,
Chinese Academy of Sciences, 100190 Beijing, China.}
\email{fanshilei@amss.ac.cn}

\author{Lingmin Liao}
\address{LAMA, UMR 8050, CNRS,
Universit\'e Paris-Est Cr\'eteil, 61 Avenue du
G\'en\'eral de Gaulle, 94010 Cr\'eteil Cedex, France}
\email{lingmin.liao@u-pec.fr}

\author{Yuefei Wang}
\address{Institute of  Mathematics, AMSS,
Chinese Academy of Sciences, 100190 Beijing, China.}
\email{wangyf@math.ac.cn}

\maketitle

\begin{abstract}
A homographic map in the field of $p$-adic numbers $\mathbb{Q}_p$
is studied as a dynamical system on $\mathbb{P}^{1}(\mathbb{Q}_p)$, the projective line over $\mathbb{Q}_p$.
If such a system admits one or
two fixed points in $\mathbb{Q}_p$, then it is conjugate to an affine dynamics whose dynamical structure
has been investigated by Fan and Fares \cite{Fan-Fares11}.
In this paper, we shall mainly solve the remaining case that  the system admits no fixed point. We shall prove that this system can be decomposed into a finite number of minimal subsystems which are topologically conjugate to each other.  All the minimal subsystems are exhibited and the unique invariant measure
for each minimal subsystem is determined.
\end{abstract}

\section{Introduction}
Let $\Qp$ be the field of $p$-adic  numbers. Throughout this paper, we let $\phi$ be a homographic map on $\Qp$ of the form
\begin{equation}\label{frac}
    \phi(x)=\frac{ax+b}{cx+d}~~~~~(a,b,c,d \in \Qp,~ad-bc\neq 0).
\end{equation}
We shall  consider $\phi$ as a one-to-one map on $\P$,  the
projective line over $\Qp$, and study the so-called homographic dynamical system
$(\P,\phi)$.

Such homographic dynamics systems on the usual field $\mathbb{C}$ of
complex numbers have relatively simple behavior. The reason is that
$\mathbb{C}$ is an algebraically closed field. But it is not the
case for the field $\mathbb{Q}_p$.

 Our main result shows that the dynamical system $(\P,\phi)$ is decomposed into
minimal subsystems. This decomposition essentially depends upon the
fact that there is one fixed point, two fixed points or no fixed
point for $\phi: \mathbb{P}^{1}(\Q_p)\to\mathbb{P}^{1}(\Q_p)$. First observe that if $c=0$, then
$\phi$ is an affine map,  of which the dynamical structure on $\Qp$ was  studied by Fan and Fares
\cite{Fan-Fares11}. So we shall assume that $c\neq 0$.  Then $\infty$ is never a fixed point and the number of fixed points of $\phi$ on $\mathbb{P}^{1}(\Q_p)$ is the same to that on $\Q_p$.  As
we shall see, if $\phi: \P \to \P$ admits one or two fixed points,
then $(\P,\phi)$ is topologically conjugate to an affine dynamics.
So we are mainly concerned with the case where $\phi$ admits no
fixed point in $\Qp$. In this case we shall consider the quadratic extension $K$ of the field $\Qp$ which contains the solutions
of $\phi(x)=x$. We first study the dynamics $(\mathbb{P}^1(K),\phi)$ then its restriction on $\mathbb{P}^1(\Qp)$.

The present study contributes to algebraic dynamical systems. See
\cite{Anashin94,Anashin98,Anashin-Uniformly-distributed02,Anashin06,Benedetto-Hyperbolic-maps,Benedetto-Components-periodic-points,
Benedetto-wandering-domain-polynomial,Hsia-periodic-points-closure,Rivera-Letelier05,Rivera-Letelier-Dynamique-rationnelles-corps-locaux,Rivera-Letelier-Espace-hyperbolique} for recent developments. See also the monographs \cite{Anashin-Khrennikov-AAD,Baker-Rumely-book,KhrennikovNilson04book,Silverman-dynamics-book} and their bibliographies therein.
The first work on affine dynamics seems to be that of Oselies and Zieschang \cite{OZ75}. They considered
the continuous automorphisms of the ring of $p$-adic integers $\Zp$ viewed
as an additive group, which are multiplication transformations $M_{a}(x) = ax$ with $a \in \Zp^{\times}$, a unit in $\Zp$, and they constructed an ergodic decomposition of
$M_{a} : \Zp^{\times}\mapsto \Zp^{\times}$, which consists of the cosets of the smallest closed subgroup
containing $a$ of the unit group $\Zp^{\times}$. These multiplication transformations were
also studied by Coelho and Parry \cite{CP11Ergodic} in order to study the distribution of
Fibonacci numbers. A full study on the ergodic decomposition of affine maps
of integral coefficients was realized by Fan, Li, Yao and Zhou \cite{FLYZ07}.  The ergodic decomposition of affine maps (with coefficients in $\Qp$) on $\Qp$ was studied by Fan and Fares \cite{Fan-Fares11}.
A general discussion was given to isometries on a valuation domain by Chabert, Fan and Fares \cite{CFF09}.
Polynomials with coefficients in $\Zp$ were studied as dynamical systems on $\Zp$ by Fan and Liao \cite{FanLiao11}.  The polynomial dynamics on a finite extension of $\Qp$ have been investigated by Fan and Liao \cite{Liao}.  Diao and Silva studied rational maps \cite{Diao-Silva}. It was shown (\cite{Diao-Silva}, Proposition 2) that rational maps are not minimal
on the whole space $\Qp.$  But in this paper we will show that there exist homographic maps (rational maps of degree one)  which are minimal on the projective line $\P$. One important reason that we can sometimes obtain the minimality on the whole space $\P$ is that the infinity point is included and considered (see Section \ref{detail}).
We remark that the minimality is equivalent to the (unique) ergodicity with respect to a (the) natural measure (see Section \ref{inv}). In the case of $\Q_p$, the natural measure is the Haar measure. While in the case of $\P$, the natural measures  are in general not the Haar measure but absolutely continuous with respect to the Haar measure, and they are determined in  Section \ref{inv}. Some other studies on the rational maps can be found in \cite{ARS13,DKM07,KM06,MR04}, in which the dynamical properties of the fixed points in $\mathbb{C}_p$ or in the adelic space are investigated, but the dynamical structure on the whole space remains unclear. In this paper, we can give a vivid picture of the dynamics of the homographic maps on whole space $\P$.

As indicated above, the dynamics of $\phi$ depends on the number of its fixed
points which are the solutions of the following equation
\begin{equation}\label{fixeq}
    \frac{ax+b}{cx+d}=x.
\end{equation}
The equation is  actually a quadratic equation $cx^2+(d-a)x-b=0$ with its
discriminant $$\Delta=(d-a)^2+4bc.$$ We distinguish three cases:
 \begin{itemize}
 \item[] {\em Case I.  $\Delta=0$}. Then $\phi$ has only one fixed point in
 $\Qp$ and
  $\phi$ is conjugate to a translation $\psi(x)=x+\alpha$ for some $\alpha\in\Qp$.
  The minimal decomposition of $\phi$ is deduced from that of  $x\mapsto x+\alpha$
  which is known in \cite{Fan-Fares11}.
 \item[] {\em Case II. $\Delta\neq 0$ and $\sqrt{\Delta}\in \Qp$}. Then
   $\phi$ has two fixed points in $\Qp$  and  $\phi$ is conjugate to a multiplication
    $x \mapsto \beta x$ for some $\beta\in\Qp$.
    The minimal decomposition of $\phi$ is deduced from that of  $x\mapsto \beta x$
  which is also known in \cite{Fan-Fares11}.
 \item[] {\em Case III.  $\Delta\neq 0$ and $\sqrt{\Delta} \notin \Qp$}.
 Then $\phi$ has no fixed point in $\Qp$.
 But $\phi$ has two fixed points in the quadratic extension $\Qp(\sqrt{\Delta})$ of $\Qp$.
\end{itemize}
Our main work is the study of the Case III (see Section \ref{detail}
for details). For $x\in \mathbb{P}^{1}(\Qp)$,  the orbit of $x$ under $\phi$ is
defined by
$$\mathcal{O}_{\phi}(x):=\{\phi^{n}(x):n\geq 0\}.$$
If $E$ is a compact $\phi$-invariant subset (i.e. $\phi(E)\subset E$),
then $(E, \phi)$ is a subsystem of $(\P,\phi)$.  The minimality of $(E, \phi)$ means that
 $E$ is equal to the closure $\overline{\mathcal{O}_{\phi}(x)}$ for
each $x\in E$.

\begin{theorem}\label{main1}
Suppose that $\phi$ has no fixed point in
$\mathbb{P}^{1}(\mathbb{Q}_p)$ and $\phi^n\neq id$ for all integers
$n>0$. Then the system $(\mathbb{P}^{1}(\mathbb{Q}_p),\phi)$ is
decomposed into a finite number of minimal subsystems. These minimal
subsystems are topologically conjugate to each other. The number of
 minimal subsystems is determined by the number
$$\lambda:=\frac{a+d+\sqrt{\Delta}}{a+d-\sqrt{\Delta}}.$$
\end{theorem}

The details of the decomposition in Theorem \ref{main1} are described in Theorems \ref{mindecunrami}-\ref{rami-decomposition-p>2} and Theorems \ref{decomp-p=2-unrami}-\ref{decompositionsqrt3}. Notice that for $p=3$, we distinguish the unramified quadratic extention from the ramified ones. While for $p=2$, we need distinguish three cases: $\Q_2(\Delta)=\Q_2(\sqrt{-3})$; $\Q_2(\Delta)=\Q_2(\sqrt{\pm 2}), \Q_p(\sqrt{\pm 6})$; and $\Q_2(\Delta)=\Q_2(\sqrt{-1}), \Q_2(\sqrt{3})$.

 We also prove that the minimal subsystems are
conjugate to adding machine on  an odometer.  Let $(p_s)_{s\geq 1}$ be a sequence of
positive integers such that $p_s|p_{s+1}$ for every $s\geq 1$. We
denote by $\mathbb{Z}_{(p_s)}$ the inverse limit of
$\mathbb{Z}/p_{s}\mathbb{Z}$, which is called an odometer. The map
$x\mapsto x+1$ is called the \emph{adding machine} on
$\mathbb{Z}_{(p_s)}$.

\begin{theorem} Under the same assumption of Theorem \ref{main1}, the
minimal subsystems of $\phi$ are topologically conjugate to the
adding machine on the odometer $\mathbb{Z}_{(p_s)}$, where
$$(p_s)=(k,kp,kp^2,\cdots)$$  for some $k|(p+1)$.
\end{theorem}

The paper is organized as follows. In Section \ref{prel}, we give some preliminaries,  including the computation of the distance to
$\mathbb{Q}_p$ from a point in a quadratic extension and the determination of the intersection with  $\mathbb{Q}_p$ of a disk of a quadratic extension.
Section \ref{fixedpoints} is devoted to the minimal decomposition of a homographic map which admits fixed points in $\Q_p$. We discuss in Section \ref{DynExt} the minimal decomposition of multiplications on a finite extension of $\Q_p$ which will be used in Section \ref{detail}. The main part of this paper is Section \ref{detail}, where we give the minimal decomposition for the homographic maps without fixed point in $\Q_p$. In the last section, we determine the unique invariant measure on each minimal subsystem.

\section{Preliminaries}\label{prel}

In this section, we first present some  notation and facts
 concerning the finite extension of $\Q_p$. Then we calculate the distance to $\Q_p$ from a point in a quadratic extension of $\Q_p$. Finally, we discuss the intersection with $\Q_p$ of disks in a quadratic extension of $\Q_p$. The facts presented in this section will be useful for determining the minimal decomposition of a homographic map without fixed point in $\Q_p$.

\subsection{Finite extensions of the field of $p$-adic numbers}\label{finiteextension}
 Let us recall some  notation and facts concerning  the finite extensions of $\Q_p$. Let $K$ be a finite
extension of $\Q_{p}$ and let $d=[K:\Qp]$
denote the dimension of $K$ as a vector space over $\Qp$.   The
extended absolute value on $K$ is still denoted by $|\cdot|_{p}$.
For $x\in K^{*}:=K\setminus \{0\}$, $v_{p}(x):=-\log_{p}(|x|_{p})$
defines the valuation of $x$, with convention $v_{p}(0):=\infty$.
One can show that there exists a unique positive integer $e$ which
is called \emph{ramification index} of $K$ over $\Qp$, such that
$$v_{p}(K^*)=\frac{1}{e}\Z.$$(Sometimes, we write the image  of $K^*$ under $| \cdot |_p$ as $| K^* |_p=p^{{\mathbb{Z}}/{e}}$. )

The extension $K$ over $\Qp$ is said to
be \emph{unramified} if $e=1$, \emph{ramified} if $e>1$ and
\emph{totally ramified} if $e=d$.  An element $\pi\in K$ is called a
\emph{uniformizer} if $v_{p}(\pi)=1/e$. For convenience of notation, we
write $$ v_{\pi}(x):=e\cdot v_{p}(x)$$
 for $x\in K$ so that
$v_\pi(x)\in \mathbb{Z}$.
Let $\mathcal{O}_{K}:=\{x\in
K : |x|_p\leq 1\}$, whose elements are called  \emph{integers} of $K$. Let
$\mathcal{P}_{K}:=\{x\in K: |x|_{p}<1\}$, which is the maximal ideal
of $\mathcal{O}_{K}$. The \emph{residual class field} of $K$ is
$\mathbb{K}:=\mathcal{O}_{K}/\mathcal{P}_{K}$. Then
$\mathbb{K}=\mathbb{F}_{p^f}$,  the finite field of $p^{f}$ elements
where $f=d/e$. Let $C=\{c_{0},c_{1},\cdots,c_{p^f-1}\}$ be a fixed
complete set of representatives of the cosets of $\mathcal{P}_{K}$
in $\mathcal{O}_{K}$. Then every $x\in K$ has a unique $\pi$-adic
expansion of the form
\begin{equation}\label{piadic}
   x=\sum_{i=i_{0}}^{\infty}a_i\pi^{i}
\end{equation}
where $i_{0}\in \Z$ and $a_i\in C$ for all $i\geq i_{0}$.

Let $\mathbb{U}:=\{x\in\mathcal{O}_{K}:|x|_{p}=1\}$ be the group of
units in $\mathcal{O}_{K}$ and $\mathbb{V}:=\{x\in
\mathbb{U}: x^{m}=1 \mbox{~for some~} m\in \mathbb{N}^*\}$ be the set of roots
of unity in $\mathcal{O}_{K}$.

 For $a\in K$ and $r>0$, we define $\mathbb{D}(a,r)$ to be the open disk of radius $r$ centered at $a$, in other words
 $$\mathbb{D}(a,r):=\{x\in K: |x-a|_{p}< r \}.$$
 Similarly, $\overline{\mathbb{D}}(a,r)$ is the closed disk of radius $r$ centered at $a$:
 $$\overline{\mathbb{D}}(a,r):=\{x\in K: |x-a|_{p}\leq r \},$$
 and $\mathbb{S}(a,r)$ is the sphere of radius $r$ centered at $a$:
 $$\mathbb{S}(a,r):=\{x\in K : |x-a|_{p}= r\}.$$
We should remark that $\CD(a,r)=\D(a,r)$ when $r\notin |K^{*}|_{p}$, and this implies $\mathbb{S}(a,r)=\emptyset$.
For $r\in |K^{*}|_{p}$, it is easy to see that $\CD(a,r)=\D(a,rp^{1/e})$.

Let $\mathbb{P}^{1}(K)$ be
the projective line over $K$. Elements of $\mathbb{P}^{1}(K)$ may be
written as $[x,y]$ with $x,y \in K$ not both zero, and with the
equivalence $[x,y]=[cx,cy]$ for $c\in K^*$. The field $K$ is embedded into
$\mathbb{P}^{1}(K)$ by the map $x\mapsto [x,1]$. The couple $[1,0]$ is the
point at infinity. An element
$$\phi=\left(
      \begin{array}{cc}
        a & b \\
        c & d \\
      \end{array}
    \right)
$$
of the projective linear group \[{\rm{PGL}}(2,K)=\{\phi([x,y])=[ax+by,cx+dy]: a,b,c,d\in K \mbox{~and~} ad-bc\neq 0 \}\] defines a map on $\mathbb{P}^1(K)$ by sending $[x,y]$
to $[ax+by,cx+dy]$.
 We usually view $\mathbb{P}^{1}(K)$ as $K\cup \{\infty\}$ and abuse notation by writing
 $$\phi(x)=\frac{ax+b}{cx+d}.$$
 The chordal metric defined on  $\mathbb{P}^{1}(K)$ is analogous to the standard chordal metric on the Riemann sphere. If $P=[x_1,y_1]$ and $Q=[x_2,y_2]$ are two points in $\mathbb{P}^{1}(K)$, we define
 $$\rho(P,Q)=\frac{|x_1y_2-x_2y_1|_p}{\max\{|x_1|_{p},|y_1|_{p}\}\max\{|x_2|_{p},|y_{2}|_{p}\}}$$
 or, viewing $\mathbb{P}^{1}(K)$ as $K\cup\{\infty\}$, for $z_1,z_2 \in K\cup \{\infty\}$
 we define
 $$\rho(z_1,z_2)=\frac{|z_1-z_2|_{p}}{\max\{|z_1|_{p},1\}\max\{|z_2|_{p},1\}}  \qquad\mbox{if~}z_{1},z_{2}\in K,$$
 and
 $$\rho(z,\infty)=\left\{
                    \begin{array}{ll}
                      1, & \mbox{if $|z|_{p}\leq 1$;} \\
                      1/|z|_{p}, & \mbox{if $|z|_{p}> 1$.}
                    \end{array}
                  \right.
 $$

\begin{definition}An open (resp. closed ) $\mathbb{P}^{1}(K)$-disk  is either an open (resp. closed ) disk $\mathbb{D}$ of $K$ or the complement $\mathbb{P}^{1}(K)\setminus \overline{\mathbb{D}}$ of
a closed (resp. open) disk $\mathbb{D}$.
\end{definition}
Remark that an open $\mathbb{P}^{1}(K)$-disk is also a closed $\mathbb{P}^{1}(K)$-disk, and vice versa.
\begin{proposition}\label{disk}
Let $\phi \in {\rm{PGL}}(2,K)$ and let $\mathbb{D}$ be a $\mathbb{P}^{1}(K)$-disk. Then
$\phi(\mathbb{D})$ is also a $\mathbb{P}^{1}(K)$-disk.
\end{proposition}
\begin{proof}
Note that each $\phi\in {\rm{PGL}}(2,K)$ is a composition of transformations of the form $x\mapsto \alpha x, x\mapsto x+\beta$ or $x\mapsto 1/x$.
So it suffices to prove our conclusion for the above three transformations.

Observe that $\phi$ is a bijection from $\mathbb{P}^{1}(K)$ to itself, we have $\phi(\PK\setminus X)=\PK\setminus \phi(X)$ for any $X\subset \PK$.
 So it  suffices to obtain our conclusion for $\mathbb{D}=\mathbb{D}(a,r)$ which is a disk of $K$.
\begin{itemize}
  \item [(1)] If $\phi(x)=\alpha x$, then $ |\phi(x)-\phi(y)|_{p}=|\alpha|_p|(x-y)|_{p},~\forall x,y\in \mathbb{D}$.  So that  $$\phi(\mathbb{D}(a,r))=\mathbb{D}(a\alpha, r|\alpha|_{p}).$$
  \item [(2)] If $\phi(x)= x+\beta$, then $ |\phi(x)-\phi(y)|_{p}=|x-y|_p,~\forall x,y\in \mathbb{D}$. Thus, $$\phi(\mathbb{D}(a,r))=\mathbb{D}( a+\beta, r).$$
  \item [(3)] If $\phi(x)=1/x$, we distinguish two cases :\\
  (i) If $0\in \mathbb{D}(a,r)$, then $\D(a,r)=\D(0,r)$.   Thus $$\phi(\mathbb{D}(a,r))=\mathbb{P}^{1}(K)\setminus \overline{\mathbb{D}}(0,1/r ).$$ \\
  (ii) If $0\notin \mathbb{D}(a,r)$, then $x\in \D(a,r)$ implies that $|x|_{p}=|a|_{p}$. So
  $$|\phi(x)-\phi(y)|_{p}=\frac{|x-y|_{p}}{|xy|_{p}}=\frac{|x-y|_{p}}{|a|^{2}_{p}}~~~\forall x,y\in \mathbb{D}.$$
  Thus, $\phi(\mathbb{D}(a,r))=\mathbb{D}(a^{-1}, r|a|_{p}^{-2}).$
\end{itemize}
\end{proof}

\subsection{Square roots and quadratic extensions of $\Q_p$}\label{extensions}
We first recall the conditions under which a number in $\Q_p$ has a square root in $\Q_p$, then we present all possible quadratic extensions of $\Q_p$.

An integer $a\in \Z$ is called a\emph{ quadratic residue modulo} $p$ if the equation $x^{2}\equiv a ~(\!\!\!\!\mod p)$ has a solution
$x\in \Z$. The following lemma characterizes those $p$-adic integers which admit a square root in $\Qp$.
\begin{lemma}[\cite{Mahler81}]\label{solution} Let $a$ be a nonzero $p$-adic number with its p-adic expansion
$$ a=p^{v_{p}(a)}(a_0+a_{1}p+a_{2}p^{2}+\cdots)$$
where $1\le a_0\le p-1$ and $0\leq a_j \leq p-1 \ (j\geq 1)$ . The equation
$x^2=a$ has a solution $x\in \Qp$ if and only if
 the following conditions are satisfied
\, \indent \ \
  \item[{\rm (i)}] $v_{p}(a)$ is even;
 \, \indent \ \
  \item[{\rm (ii)}] $a_0$ is quadratic residue modulo $p$ if $p\neq 2$; or $a_1=a_2=0$ if $p=2$.
\end{lemma}
For the case where $\sqrt{\Delta}\notin \Qp$, we need to study the affine systems on some  quadratic extension of $\Qp$.
Two distinct elements $a$ and $a^{\prime}$ of $\Qp$, neither of which is $0$ or the square of a $p$-adic number, evidently produce the same extension field
$$\Qp(\sqrt{a})=\Qp(\sqrt{a^{\prime}})$$
if and only if the quotient $a/a^{\prime}$ is the square of a $p$-adic number. Actually, there are $7$ distinct quadratic extensions of $\Q_{2}$, and for $p\geq 3$
there are $3$ distinct extensions of $\Qp$ as the following lemma shows.

\begin{lemma}[\cite{Mahler81}]
There are exactly $7$ distinct quadratic extensions of $\Q_2$.
They are represented respectively by
$$\Q_2(\sqrt{-1}),\Q_2(\sqrt{2}),\Q_2(\sqrt{-2}),\Q_2(\sqrt{3}),\Q_2(\sqrt{-3}),\Q_2(\sqrt{-6}),\Q_2(\sqrt{6}).$$
If $p\geq 3$, then $\Q_p$ has exactly $3$ distinct quadratic
extensions: 
$$\Qp(\sqrt{N_p}),\Qp(\sqrt{p}),\Qp(\sqrt{pN_p}),$$
where $N_p<p$ is the smallest positive integer which is not a quadratic residue modulo $p$.
\end{lemma}
The following lemma shows that each field $\Qp$ admits exactly one unramified quadratic extension.
\begin{lemma}[\cite{Mahler81}]\label{unramifiedextension}
For the field $\Q_{2}$ , there is  exactly one unramified quadratic extension $\Q_{2}(\sqrt{-3})$.
For the field $\Q_{p}$  where $p\geq 3$, there is also exactly one unramified quadratic extension $\Q_{p}(\sqrt{N_{p}})$.
\end{lemma}

\subsection{Distance $d(x,\Qp)$ from a point in a quadratic extension to $\Q_p$}\label{distances}
Let $K$ be a quadratic extension of $\Qp$. For $x\in K$ and $S\subset K$, denote by
$$d(x,S)=\inf_{y\in S}|x-y|_{p}$$
the distance from $x$ to $S$. In the following, we will compute the distance from a point $x\in K$ to the set $\Q_p$.
\begin{proposition}\label{prop3}
Let $x_{1},x_{2}\in\Qp$ be two nonzero $p$-adic numbers. If $\sqrt{\frac{x_{1}}{x_{2}}}\in \Qp$,
then $$d(\sqrt{x_{1}},\Qp)=\left|\sqrt{\frac{x_{1}}{x_{2}}}\right|_{p}\cdot d(\sqrt{x_{2}},\Qp).$$
\end{proposition}
\begin{proof}
Let $t=\frac{x_1}{x_2}$. We have
$$d(\sqrt{x_{1}},\Qp)=\inf_{y\in \Qp}|\sqrt{tx_{2}}-y|_{p}=|\sqrt{t}|_{p}\cdot(\inf_{y\in \Qp}|\sqrt{x_{2}}-\frac{y}{\sqrt{t}}|_{p}).$$
Since $\sqrt{t}\in \Qp$, the last infimum  is equal to $d(\sqrt{x_{2}},\Qp)$.
\end{proof}
\begin{proposition}\label{distanceNp}
Let $p\geq 3 $ be a prime number.
We have $$d(\sqrt{N_{p}},\Qp)=1,~d(\sqrt{pN_{p}},\Qp)=d(\sqrt{p},\Qp)=p^{-1/2}.$$
\end{proposition}
\begin{proof}
First the fact $|N_{p}|_{p}=1$ implies $|\sqrt{N_{p}}|_{p}=1$. Let $k\in \{0,1,\cdots,p-1\}$.
Then
\begin{eqnarray}\label{fact1} |\sqrt{N_{p}}\pm k|_{p}\leq 1.\end{eqnarray} Since $N_{p}$ is not a quadratic residue modulo $p$, we have
$$N_{p}-k^{2}\not\equiv 0 \mod p.$$
 It follows that $|N_{p}-k^{2}|_{p}=1$. Then from $N_{p}-k^{2}=(\sqrt{N_{p}}+k)(\sqrt{N_{p}-k})$, and the fact (\ref{fact1}), we deduce $$|\sqrt{N_{p}}+k|_{p}=|\sqrt{N_{p}}-k|_{p}=1 .$$
Hence for $x\in \Qp$, we have
$$|x-\sqrt{N_{p}}|_{p}=\left\{
      \begin{array}{ll}
        |x|_{p}, & \mbox{if $|x|_{p}>1$,} \\
        1, & \mbox{if $|x|_{p}\leq 1$.}
      \end{array}
    \right.$$
So  $d(\sqrt{N_{p}},\Qp)=1$.

Notice that $|\sqrt{p}|_{p}=|\sqrt{pN_{p}}|_{p}=p^{-1/2}$. For $x\in \Qp$, we have
$$|x-\sqrt{p}|_{p}=\left\{
      \begin{array}{ll}
        |x|_{p}, & \mbox{if $|x|_{p}\geq 1$,} \\
        p^{-1/2}, & \mbox{if $|x|_{p}<1$;}
      \end{array}
    \right.$$

$$|x-\sqrt{pN_{p}}|_{p}=\left\{
      \begin{array}{ll}
        |x|_{p}, & \mbox{if $|x|_{p}\geq 1$,} \\
        p^{-1/2}, & \mbox{if $|x|_{p}<1$.}
      \end{array}
    \right.$$
So  $d(\sqrt{pN_{p}},\Qp)=d(\sqrt{p},\Qp)=p^{-1/2}.$
\end{proof}

As a corollary of   Propositions \ref{prop3} and \ref{distanceNp}, we have
\begin{corollary}\label{p>2distance}
Let $p\geq 3$ be a prime number and $x\in \Qp $. If $\sqrt{x}\notin \Qp$, then
$$d(\sqrt{x},\Qp)=|\sqrt{x}|_{p}.$$
\end{corollary}
\begin{corollary}\label{sqrtx+yp>2}
Let $p\geq 3$ be a prime number and $x\in \Qp $ with $\sqrt{x}\notin \Qp$. If
$|y|_{p}\leq |\sqrt{x}|_{p}$, then
$$|\sqrt{x}-y|_{p}=|\sqrt{x}|_{p}.$$
\end{corollary}
\begin{proof}
Since $|y|_{p}\leq |\sqrt{x}|_{p}$, it follows that
$|\sqrt{x}-y|_{p}\leq |\sqrt{x}|_{p}$.
Also, we have
$$|\sqrt{x}-y|_{p}\geq d(\sqrt{x}, \overline{D}(0,|\sqrt{x}|_{p}))=
d(\sqrt{x}, \Qp)=|\sqrt{x}|_{p}.$$
So
$$|\sqrt{x}-y|_{p}=|\sqrt{x}|_{p}.$$
\end{proof}

%
%

\begin{proposition}\label{prop4}
For $p=2$, we have \\
\indent \mbox{\rm 1)} \  $d(\sqrt{-3},\Q_{2})=1/2$;\\
\indent \mbox{\rm 2)} \  $d(\sqrt{i},\Q_{2})=\sqrt{2}/2 \ $ for $i=2,-2,6 \mbox{~or~} -6$;\\
\indent \mbox{\rm 3)} \  $d(\sqrt{i},\Q_{2})=\sqrt{2}/2 \ $ for $i=-1\mbox{~or~} 3.$
\end{proposition}
\begin{proof}
In order to determine $d(\sqrt{i},\Q_2)$, we are going to compute $|x-\sqrt{i}|_{2}$ for all $x\in \Q_{2}$ and for $i=-3,2,-2,6,-6,-1$ or $3$.

\indent \mbox{\rm 1)} \  Assume $i=-3$. We claim that
for all $x\in \Q_{2}$,
\begin{equation}\label{sqrt-3}
|x-\sqrt{-3}|_{2}\geq |1-\sqrt{-3}|_{2}=1/2.
\end{equation}
In fact, since $(1-\sqrt{-3})^{3}=-8$, we get the equality $|1-\sqrt{-3}|_{2}=1/2$ in (\ref{sqrt-3}). To prove the inequality in (\ref{sqrt-3}), we observe that $|\sqrt{-3}|_{2}=1$, which follows from $|-3|_{2}=1$.
 We distinguish three cases.\\
(i) If $|x|_{2}>1$, then $|x-\sqrt{-3}|_{2}=|x|_{2}>1$.\\
(ii) If $|x|_{2}<1$, then $|x-\sqrt{-3}|_{2}=|\sqrt{-3}|_{2}=1$.\\
(iii) If $|x|_{2}=1$, we write $x=1+2t$ for some $t\in \mathbb{Z}_{2}$.  A simple calculation gives
   $$   (1+2t-\sqrt{-3})^{3}= -8+12t(t+1)(1-\sqrt{-3})+8t(t^{2}-3).$$
Since $|t(t+1)|_{2}\leq 1/2$, $ |t(t^{2}-3)|_{2}\leq 1/2$ and $|1-\sqrt{-3}|_{2}=1/2$, we get
 $$|(x-\sqrt{-3})^{3}|_{2}=|(1+2t-\sqrt{-3})^{3}|_{2}=1/8.$$
So  $|x-\sqrt{-3}|_{2}=1/2$.\\
\indent \mbox{\rm 2)} \  Assume $i=2,-2,6$ or $-6$. We have $|\sqrt{i}|_{2}=\sqrt{2}/2$. So
 $$|x-\sqrt{i}|_{2}=\left\{
       \begin{array}{ll}
         \sqrt{2}/2, & \mbox{if $|x|_{2}\leq 1/2$;} \\
         |x|_{2}, & \mbox{if $|x|_{2}\geq 1$.}
       \end{array}
     \right.$$
Thus $d(\sqrt{i},\Q_{2})=\sqrt{2}/2$.\\
\indent \mbox{\rm 3)} \  Assume $i=-1~\mbox{or}~3$. The fact $|i|_{2}=1$ implies $|\sqrt{i}|_{2}=1$.  From the facts
$(1-\sqrt{i})^{2}=1+i -2\sqrt{i}$ and $|1+i|_{2}\leq 1/4$, we get
$$|1-\sqrt{i}|_{2}=\sqrt{2}/2.$$
We claim that $|x-\sqrt{i}|_{2}\geq \sqrt{2}/2$ for all $x\in \Qp$.
We distinguish three cases.\\
(a) If $|x|_{2}>1$, then $|x-\sqrt{i}|_{2}=|x|_{2}>1$.\\
(b) If $|x|_{2}<1$, then $|x-\sqrt{i}|_{2}=|\sqrt{-3}|_{2}=1$.\\
(c) If $|x|_{2}=1$,  we write $x=1+2t$ for some $t\in \mathbb{Z}_{2}$.  Observe that
$$(1+2t-\sqrt{i})^{2}=1+i+4t(t+1)-2(1+2t)\sqrt{i}.$$
From the facts $|1+i|_{2}\leq 1/4$ and $|4t(t+1)|_{2}\leq 1/8$, we get
$$|1+2t-\sqrt{i}|_{2}=\sqrt{2}/2.$$
Therefore $d(\sqrt{i},\Q_{2})=\sqrt{2}/2$.
\end{proof}

%
%
By Propositions \ref{prop3} and \ref{prop4}, we have the following corollary.
\begin{corollary}\label{distancep=2}
Let  $x$ be a number of $\Q_{2}$ such that $\sqrt{x}\not\in \Q_{2}$.
\begin{itemize}
  \item[\rm{(1)}] If $\Q_{2}(\sqrt{x})=\Q_{2}(\sqrt{-3})$, then $$d(\sqrt{x},\Q_2)=\frac{1}{2}|\sqrt{x}|_{2}.$$
  \item[\rm{(2)}] If $\Q_{2}(\sqrt{x})=\Q_{2}(\sqrt{i}),i=2,-2,6,\mbox{~or~}-6$,
  then $$d(\sqrt{x},\Q_2)=|\sqrt{x}|_{2}.$$
  \item[\rm{(3)}] If $\Q_{2}(\sqrt{x})=\Q_{2}(\sqrt{i}),i=-1\mbox{~or~} 3$, then
   $$d(\sqrt{x},\Q_2)=\frac{\sqrt{2}}{2} |\sqrt{x}|_{2}.$$
\end{itemize}
\end{corollary}

We can obtain more information from the proof of Proposition \ref{prop4}.
\begin{proposition}\label{x-sqrt}
Let $x\in \Q_2$ with $|x|_{2}=1$. Then we have \\
\indent\rm{(1)} $|x-\sqrt{-3}|_2=1/2$,\\
 \indent\rm{(2)} $|x-\sqrt{-1}|_2=|x-\sqrt{3}|_2=\sqrt{2}/2.$

\end{proposition}
\begin{proof}
\indent\rm{(1)}  Since $(-1-\sqrt{-3})^{3}=8$ and $(1-\sqrt{-3})^{3}=-8$, $$|-1-\sqrt{-3}|_{2}=|1-\sqrt{-3}|_2=1/2.$$
Since $|x|_{2}=1$, it follows that either $|x-1|_2 < 1/2$ or $|x+1|_2<1/2$. Without loss of generality, we may suppose
$|x+1|_2 < 1/2$.  Then we have
$$|x-\sqrt{-3}|_{2}=|x+1-1-\sqrt{-3}|_{2}=1/2.$$

\indent\rm{(2)} From the proof of Proposition \ref{prop4}, we have $|1-\sqrt{-1}|_2=|1-\sqrt{3}|_{2}=\sqrt{2}/2.$
Since $|x|_{2}=1$, $|x-1|_{2}\leq 1/2$.
So $$|x-\sqrt{-1}|_{2}=|x+1-1-\sqrt{-1}|_{2}=\sqrt{2}/2.$$
Applying the same arguments to $\sqrt{3}$ instead of $\sqrt{-1}$, we obtain the conclusion.

\end{proof}

\subsection{Intersection with $\Q_p$ of disks in a quadratic extension}

By the facts presented in the previous subsections, we can determine the intersection with $\Q_p$ of disks in a quadratic extension of $\Q_p$, which in some sense shows the topological structures of the two different kinds of (ramified and unramified)  quadratic extensions of $\Qp$.

\begin{lemma}\label{unrami-number}
Let $K$ be an unramified quadratic extension of $\Qp$. 
Let $\CD=\CD(a,p^{m})$ be a closed disk of radius $p^{m}$ with $m\in \Z$ such that
$\CD\cap \Qp\neq \emptyset.$
Then $\CD$ consists of $p^{2}$ closed disks of radius $p^{m-1}$, and there are
$p$ such disks which intersect $\Qp$.
\end{lemma}
\begin{figure}
  \centering
  \includegraphics[width=0.6\textwidth]{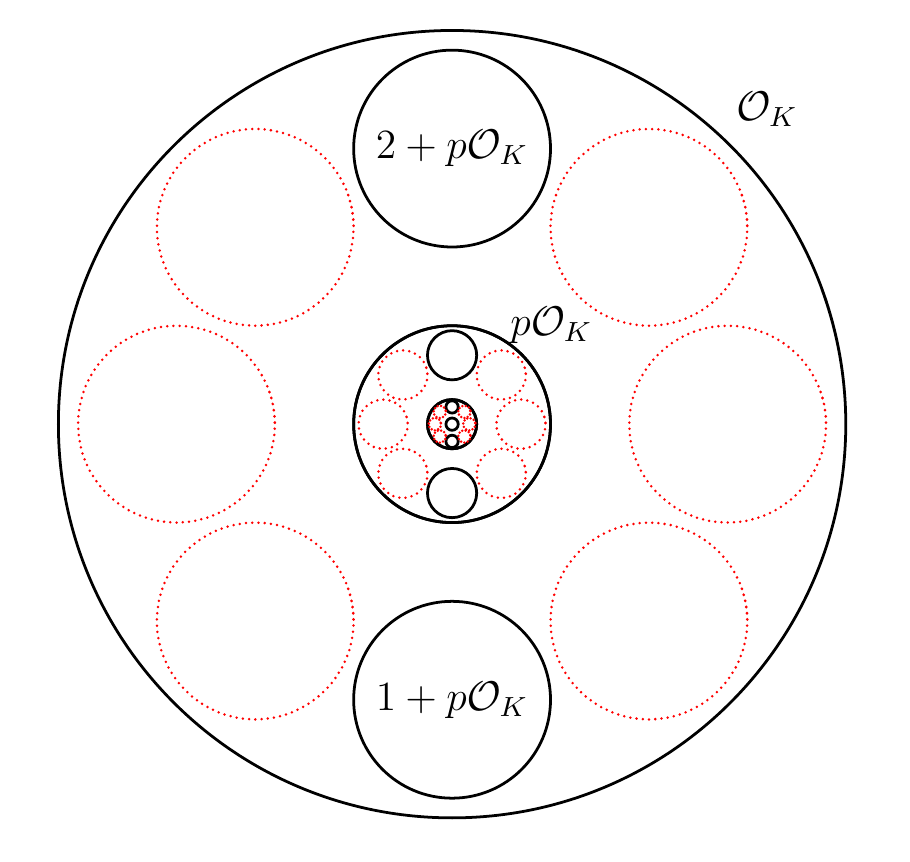}\\
  \caption{The topological structure of the unramified quadratic extension of $\Q_{3}$.}
\label{unraimipic}
\end{figure}
\begin{proof}
Without loss of generality, we may suppose that $0\in \CD$. So $\CD=\CD(0,p^{m})$.
We distinguish the following two cases.\\
\indent {\rm (1)} Assume $p\geq 3$. Then $K=\Qp(\sqrt{N_p})$. Each point $z\in \CD$ can be written as
$z=x+y\sqrt{N_p}$ with $x,y\in D(0,p^m)$ (see Theorem 5 of $\S 6$  in \cite{Mahler81}). Let $x=\sum_{k=-m}^{\infty}x_{k}p^{k}$ and $y=\sum_{k=-m}^{\infty}y_{k}p^{k}$, where $x_{k},y_{k}\in \{0,1,\cdots p-1\}$.
We decompose $\CD$ as $$\CD=\bigsqcup_{x_m,y_m\in \{0,1,\cdots,p-1\}}\CD\big(x_mp^{-m}+y_{m}p^{-m}\sqrt{N_{p}}, \ p^{m-1}\big).$$

If $y_m\neq 0$, by Proposition \ref{distanceNp}, we get $$\CD(x_{m}p^{-m}+y_{m}p^{-m}\sqrt{N_{p}},\ p^{m-1})\cap \Qp=\emptyset.$$
So $\CD(0,p^{m-1}),\CD(p^{-m},p^{m-1}),\cdots,\CD \big((p-1)p^{-m},p^{m-1}\big)$ are the $p$ closed disks of radius $p^{m-1}$, all of  which intersect  $\Qp$.
\\
\indent {\rm (2)} Assume $p=2$. By Lemma \ref{unramify1}, we have $K=\Q_{2}(\sqrt{-3})$.
Each point $z\in\CD$ can be written as  $$z=x+y\frac{-1+\sqrt{-3}}{2}$$ with $x,y\in D(0,p^m)$, (still see Theorem 5 of $\S 6$  in \cite{Mahler81}).
By similar arguments to Case (1), we obtain the result.
\end{proof}
The  topological structure of  the unramified quadratic extensions  of $\Q_{3}$  is depicted in Figure \ref{unraimipic}.
 Here a circle of real line represents a disk which intersects $\Q_{3}$ and a circle of dotted line represents a disk which is disjoint from $\Q_{3}$.

\begin{lemma}\label{rami-number}
Let $K$ be a ramified quadratic extension of $\Qp$. 
Let $\CD=\CD(a,p^{m/2})$ be a closed disk of radius $p^{m/2}$ with $m\in \Z$ such that
$\CD\cap \Qp\neq \emptyset.$
Then $\CD$ consists of $p$ closed disks of radius $p^{(m-1)/2}$.
If $m$ is even, then all such disks of radius $p^{(m-1)/2}$ intersect $\Qp$; if $m$ is odd, then there is
only one such disk which intersects $\Qp$.
\end{lemma}
\begin{proof}
Without loss of generality, we may suppose  $0\in \CD$. So $\CD=\CD(0,p^{m/2})$.

\indent {\rm (1)} Assume $p\geq 3$. Then $K=\Qp(\sqrt{i})$, where $i=p$ or $p\cdot N_{p}$.
Observe that $\sqrt{i}$ is a uniformizer of $K$.
Each point $z\in \CD$ can be written as
$$z=\sum_{k=-m}^{\infty}z_{k}\cdot(\sqrt{i})^{k}$$
where $z_{k}\in \{0,1,\cdots,p-1\}.$
We decompose $\CD$ as $$\CD=\bigsqcup_{z_m\in\{0,1,\cdots,p-1\}}\CD\big(z_{m}\cdot(\sqrt{i})^{m}, \ p^{(m-1)/2}\big).$$

If $m$ is odd, then $\CD(0,p^{(m-1)/2})$ is the unique disk of radius $p^{(m-1)/2}$ which is contained in $\CD$ and intersects $\Qp$.

If $m$ is even, then $z_{m}\cdot(\sqrt{i})^{m}\in \Qp$. So \[\CD(z_{m}\cdot(\sqrt{i})^{m},p^{(m-1)/2})\cap \Qp \neq \emptyset, \quad \forall z_m\in \{0,1,\cdots,p-1\}.\]
\indent {\rm (2)} Assume  $K=\Q_{2}(\sqrt{i})$, where $i=2,-2,6$ or $-6$. We conclude by the same arguments as in Case (1).\\
\indent {\rm (3)} Assume   $K=\Q_{2}(\sqrt{i})$, where $i=-1\mbox{~or~}3$.

If $m$ is even, then $\CD(0,p^{m/2})=\CD(0,p^{(m-1)/2})\cup\CD(p^{m/2},p^{(m-1)/2})$ is a union of two balls. Both balls $\CD(0,p^{(m-1)/2})$ and $ \CD(p^{m/2},p^{(m-1)/2})$
intersect  $\Q_2$.

If $m$ is odd, observe that $|\sqrt{i}-1|_{2}=\sqrt{2}/2$,  then $$\CD(0,p^{m/2})=\CD\big(0,p^{(m-1)/2})\cup\CD((\sqrt{i}-1) p^{(m+1)/2}, \ p^{(m-1)/2}\big),$$
is still a union of two balls.
By Corollary \ref{distancep=2}, we have $$\CD((\sqrt{i}-1) p^{(m+1)/2},\ p^{(m-1)/2})\cap \Q_2=\emptyset.$$
Thus only one ball intersets $\Q_2$.
\end{proof}
The  topological structure of  the ramified quadratic extensions  of $\Q_{3}$  is depicted in Figure \ref{raimipic}.
Here a circle of real line represents a disk which intersects  $\Q_{3}$ and a circle of dotted line represents a disk which is disjoint from $\Q_{3}$.

\begin{figure}
  \centering
  \includegraphics[width=0.6\textwidth]{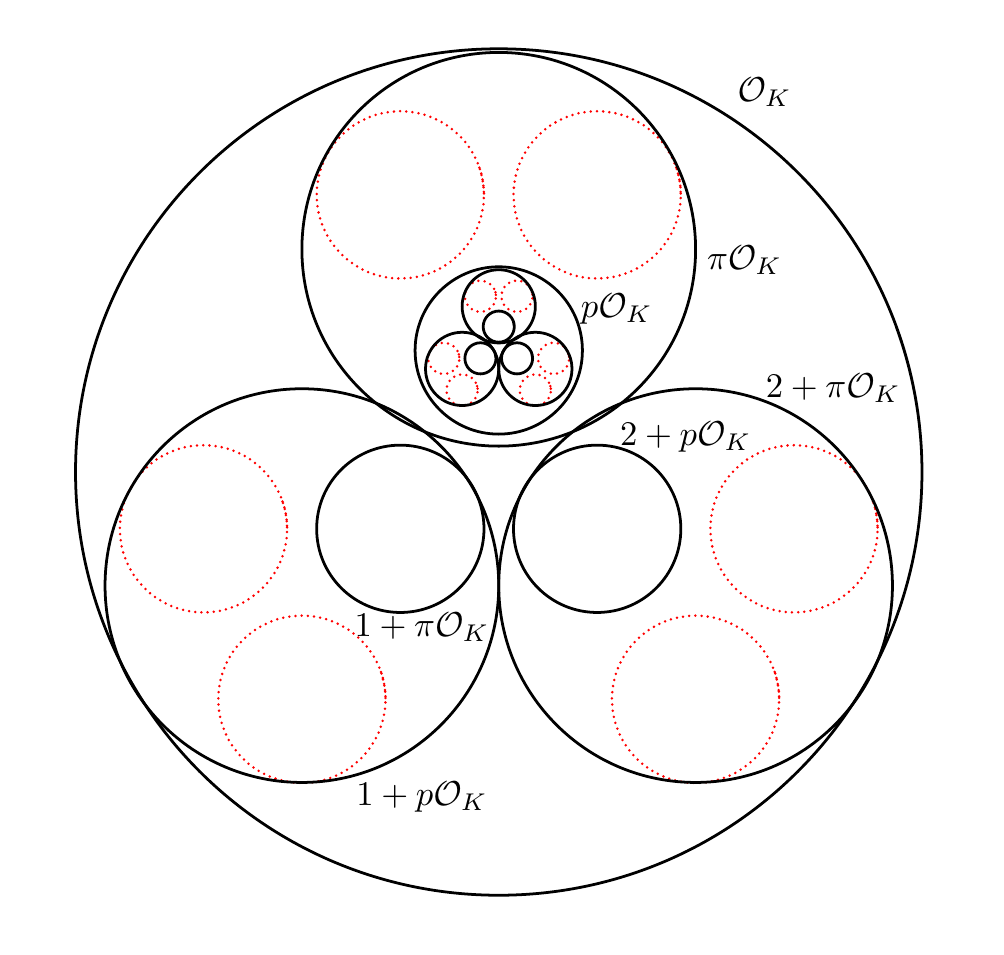}\\
  \caption{The  topological structure of  the ramified quadratic extensions  of $\Q_{3}$}\label{raimipic}
\end{figure}

\medskip
\section{Minimal decomposition when $\phi$ admits fixed points}\label{fixedpoints}

In this section, we give the minimal decomposition of a homographic map $\phi$ when $\phi$ admits one or two fixed points in $\Q_p$. Such a map is
conjugate to an affine map. Then we can apply the results obtained
in \cite{Fan-Fares11}.

 For $a\in \Qp$ and $r\in
|\Qp|_{p}:=p^{\Z}$, denote
$$\overline{D}(a,r):=\{x\in \Qp : |x-a|_{p}\leq r\};$$
$$S(a,r):=\{x\in \Qp : |x-a|_{p}= r\}$$
which are  closed disk and sphere in $\Qp$.  Recall that disk and sphere in
$K$ are differently denoted by $\overline{\mathbb{D}}(a,r)$ and
$\mathbb{S}(a,r)$.

The following notation will also be used. Let $U:=\{x\in \Qp:
|x|_{p}=1\}$ be the group of units in $\Q_p^*$ and $V:=\{x\in \Qp:
 x^{p-1}=1 \}$ be the group of the roots of
unity in $\Qp$. Set
$$s_p=1~~\mbox{if}~~p\geq 3, ~ \text{ and } ~s_p=2 ~\mbox{if}~~p=2.$$ For a unit
$a\in U$, let
$$\delta(a)=\inf\{n\geq 1:v_{p}(a^n-1)\geq s_p\}, \quad v_{0}(a)=v_p(a^{\delta(a)}-1).$$

\subsection{One fixed point: $\Delta=0$}
Assume $\Delta=0$.  Then  $\phi$ has only one fixed point
$x_{0}=\frac{a-d}{2c}\in \Qp$ and it is conjugate to the translation
$$
    \psi(x) = x + \alpha \quad \mbox{\rm with}\ \ \alpha=
    \frac{2c}{a+d}.
$$
More precisely $g\circ \phi\circ g^{-1}= \psi$ where
$$g(x)=\frac{1}{x-x_0}, \quad \text{and thus} \quad
g^{-1}(x)=\frac{x_0x+1}{x}.
$$

\begin{theorem}
Assume that $\Delta=0$. Then $x_0=\frac{a-d}{2c}\in \Qp$ is the
unique fixed point of $\phi$, and we have  
\begin{itemize}
\item [{\rm (1)}] $\mathbb{P}^{1}(\Qp)\setminus \overline{D}(x_{0}, p^{-1}|\alpha|_{p}^{-1})$ is a minimal  component of $\phi$,
\item [{\rm (2)}] for any $r=p^{m}\in |\Qp|_{p}$ with $m<v_{p}(\alpha)$, the sphere $S(x_0,r)$ consists of $p^{v_{p}(\alpha)-m-1}(p-1)$ minimal components,
and each minimal  component is a disk with radius $r^{2}|\a|_p$.
\end{itemize}

\end{theorem}
\begin{proof}
 By Theorem 4.1 in
\cite{Fan-Fares11}, the disk
  $\overline{D}(0,|\alpha|_{p})$ is a minimal component of  $\psi$ and for any $\rho=p^{-m}\in |\Qp|_{p}$ with $m<v_{p}(\alpha)$, and  the sphere $S(0,\rho)$ consists of $p^{v_{p}(\alpha)-m-1}(p-1)$ minimal components of $\psi$, with each component being a ball of radius $|\alpha|_{p}$.

By Proposition \ref{disk},
$$g^{-1}(\overline{D}(0,|\alpha|_{p}))=\mathbb{P}^{1}(\Qp)\setminus \overline{D}(x_{0},p^{-1}|\alpha|_{p}^{-1}).$$
Since $g$ is a conjugation,  $\mathbb{P}^{1}(\Qp)\setminus \overline{D}(x_{0},p^{-1}|\alpha|_{p}^{-1})$ is a minimal  component of $\phi$.

Let $w\in \Qp$ with $|w|_{p}> |\alpha|_{p}$. By Proposition \ref{disk},  we have
$$g^{-1}(\overline{D}(w,|\alpha|_{p}))=\overline{D}(x_{0}+1/w,|\alpha|_{p}/|w|_{p}^{2}).$$
Noticing that $\overline{D}(w,|\alpha|_{p})$ is a minimal ball of $\psi$ in $S(0,\rho)$ with $\rho=|w|_{p}$ and  $\overline{D}(x_{0}+1/w,|\alpha|_{p}/|w|_{p}^{2})$ is a ball in  $S(x_0,r)$ with $r=|w|_p^{-1}$, we obtain the second assertion of the theorem.
\end{proof}

\subsection{Two fixed points: $\Delta\not=0$ and $\sqrt{\Delta} \in \Qp$}
Assume  $\Delta\neq 0$ and  $\sqrt{\Delta}\in \Qp$. Then $\phi$
has two fixed points
$x_{1}=\frac{a-d+\sqrt{\Delta}}{2c},x_{2}=\frac{a-d-\sqrt{\Delta}}{2c}\in
\Qp$ and it is conjugate to the multiplication
$$
   \psi(x) = \lambda x  \quad {\rm with} \ \
   \lambda:=\frac{a+d+\sqrt{\Delta}}{a+d-\sqrt{\Delta}}.
$$
More precisely $g\circ \phi \circ g^{-1} =\psi$ where
$$g(x)=\frac{x-x_2}{x-x_1}, \quad \text{and thus} \quad
g^{-1}(x) = \frac{x_1 x-x_2}{x-1}.$$

We  distinguish four cases. The first three  cases are simple.
\begin{itemize}
\item [{\rm (a)}]
$|\lambda|_{p}>1$. Then $x_{1}$~is an attracting fixed point and for
each $x\neq x_{2}$ we have $\lim_{n\rightarrow
\infty}\phi^{n}(x)=x_{1}.$
\item [{\rm (b)}] $|\lambda|_{p}<1$.  Then $x_{2}$~is an attracting fixed point and for
each $x\neq x_{1}$ we have $\lim_{n\rightarrow
\infty}\phi^{n}(x)=x_{2}.$
\item [{\rm (c)}] $\lambda\in V$.  Let $\ell$ be the order of $\lambda$, i.e. the
least integer such that $\lambda^{\ell}=1$.
  Then every point is in a periodic orbit with period $\ell$.
\item [{\rm (d)}] $\lambda\in U\setminus V$. We can rephrase Theorem 4.2 of \cite{Fan-Fares11} as follows.
\end{itemize}
Recall the definitions of $\delta(\cdot)$ and $v_0(\cdot)$ at the beginning of this section.
\begin{theorem}
Assume that $\Delta\neq0$, $\sqrt{\Delta}\in \Qp$ and $\lambda \in
U\setminus V$.
Let $r_0=|x_1-x_2|_{p}$.
\begin{itemize}
  \item[\rm{(1)}] For $x,y\in \mathbb{P}^{1}(\Qp)$ we have $\overline{\mathcal{O}_{\phi}(x)}=\overline{\mathcal{O}_{\phi}(y)}$ if and only if $|\frac{x-x_2}{x-x_1}|_{p}=|\frac{y-x_2}{y-x_1}|_{p}$ and $\frac{(x-x_2)(y-x_1)}{(x-x_1)(y-x_2)}$ is in the subgroup of $(\Z/p^{v_{0}(\lambda)}\Z)^{\times} $ generated by $\lambda$.
  \item[\rm{(2)}] One can decompose $\P\setminus (\overline{D}(x_1,r_0/p)\cup \overline{D}(x_2,r_0/p) )$ into $(p-1)p^{v_0(\lambda)-1}/\delta(\lambda)$ minimal components; and for $r \in |\Qp^*|_{p}$ with $r<r_0$, each of both spheres $S(x_1,r)$ and $S(x_2,r)$ consists of $(p-1)p^{v_0(\lambda)-1}/\delta(\lambda)$ minimal components.
\end{itemize}
\end{theorem}
For the proof we refer the reader to \cite{Fan-Fares11}. By \cite{Fan-Fares11}, we can also show that each minimal component in the above decomposition  is conjugate to  the adding machine on the odometer $\mathbb{Z}_{(p_s)}$ with  $$(p_s)=(\delta(\lambda),\delta(\lambda)p,\delta(\lambda)p^2,\cdots).$$

\section{Dynamics of multiplications on a finite extension of $\Qp$ }\label{DynExt}
In order to study the case $\sqrt{\Delta}\notin \Qp$, we need to consider the quadratic extension $\Qp(\sqrt{\Delta})$. Actually, the dynamics $(\mathbb{P}^{1}(\Qp(\sqrt{\Delta})),\phi)$ is conjugate to a multiplication $x\mapsto \lambda x$ on $\mathbb{P}^{1}(\Qp(\sqrt{\Delta}))$ (see Section \ref{detail}). For this reason, we present
a general method to study polynomial dynamics on a finite extension of $\Qp$.

Recall that $K$ is a finite extension of $\Qp$, $\mathcal{O}_{K}$ is the integral ring of $K$ and $\pi$ is a uniformizer of $K$(see Section \ref{finiteextension}).
Let $F \in \mathcal{O}_{K}[x]$ be a polynomial of integral
coefficients. The dynamics of $(\mathcal{O}_{K}, F)$ can be
described by its induced finite dynamics on
$\mathcal{O}_{K}/\pi^n\mathcal{O}_{K}$ (refer to  \cite{CFF09}). The idea comes from a paper of
Desjardins and Zieve \cite{DZunpu} in which the authors dealt with the polynomial dynamics on
$\Z/p^n\Z$. This idea allowed Fan and Liao \cite{FanLiao11} to give a
decomposition theorem for any polynomial in $\mathbb{Z}_p[x]$. The polynomials in a finite extension $K$ of $\Qp$ are examined by Fan and Liao \cite{Liao}.

We first present the method for general polynomials and then adopt it to
multiplications by giving more details.
\subsection{Induced dynamics of polynomial on $\mathcal{O}_{K}/\pi^n\mathcal{O}_{K}$ }
Let $F \in\mathcal{O}_{K}[x]$ be a polynomial of integral
coefficients.  For a positive integer $n\geq 1$, denote by $F_{n}$ the induced mapping of $F$ on $\mathcal{O}_{K}/\pi^n\mathcal{O}_{K}$, defined by
$$F_n(x \!\!\!\mod \pi^n)=F(x)\!\! \mod \pi^n.$$
 Many properties of the dynamics $F$ are linked to those of
$F_n$. One is the following.

\begin{lemma}[\cite{Anashin94, Anashin-Uniformly-distributed02,CFF09}]\label{minimal-part-to-whole}
Let $F\in \O[x]$ and $E\subset \O$ be a compact  $F$-invariant set.
Then $F: E\to E$ is minimal if and only if $F_n: E/\pi^n\O \to
E/\pi^n\O$ is minimal for each $n\ge 1$.
\end{lemma}
Assume that $\sigma=(x_1,\cdots,x_k)\subset\mathcal{O}_{K}/\pi^n\mathcal{O}_{K}$ is a cycle of $F_n$ of length $k$, that is
$$ F_n(x_1)=x_2, \cdots, F_n(x_i)=x_{i+1}, \cdots, F_n(x_k)=x_1.$$
Such a $\sigma$ is also called a $k$-cycle at level $n$.  Let
$$X:=\bigsqcup_{i=1}^{k}X_i~\mbox{where}~~X_{i}:=\{x_i+\pi^nt+\pi^{n+1}\mathcal{O}_{K};t\in C\}\subset\mathcal{O}_{K}/\pi^{n+1}\mathcal{O}_{K},$$
where $C=\{c_{0},c_{1},\cdots,c_{p^f-1}\}$ is a fixed
complete set of representatives of the cosets of $\mathcal{P}_{K}$
in $\mathcal{O}_{K}$ (see Section \ref{finiteextension}).
Then
$$F_{n+1}(X_i)\subset X_{i+1}~~(1\leq i\leq k-1)~\mbox{and}~F_{n+1}(X_k)\subset X_1.$$

Let $G:=F^k$ be the $k$-th iterate of $F$. For $x\in \bigcup\limits_{x_{i}\in\sigma}\big(x_{i}+\pi^n\mathcal{O}_{K}\big)$, denote
\begin{eqnarray}
& & a_n(x):=G^{\prime}(x)=\prod_{j=0}^{k-1}F^{\prime}(F^{j}(x)),\label{def-an}\\
&&b_n(x):=\frac{G(x)-x}{\pi^n}=\frac{F^k(x)-x}{\pi^n}.\label{def-bn}
\end{eqnarray}
The $1$-order Taylor expansion of $G$ at $x$
$$G(x+\pi^n t)\equiv G(x)+G^{\prime}(x)\pi^{n}t~~\ ({\rm mod} \ \pi^{2n}),$$
implies
\begin{eqnarray}
& &G(x+\pi^n t)\equiv x+\pi^nb_{n}(x)+\pi^na_n(x)t~~\ ({\rm mod} \ \pi^{2n}).\label{tayor}
\end{eqnarray}
We define an affine map on $\O/ \pi \O$ by
\begin{eqnarray}\label{linearization}
\Phi(x,t)=b_n(x)+a_n(x)t~~\ ({\rm mod} \ \pi),~(t\in \mathcal{O}_{K}/\pi\mathcal{O}_{K}).
\end{eqnarray}

An important consequence of the formula (\ref{tayor}) is that $G_{n+1}:X_i\mapsto X_i$ is conjugate to the linear map
$$\Phi(x_{i},\cdot):\mathcal{O}_{K}/\pi\mathcal{O}_{K}\mapsto \mathcal{O}_{K}/\pi\mathcal{O}_{K}.$$
We  call it the \emph{linearization} of $G_{n+1}:X_{i}\mapsto X_{i}.$

For any $x\in \mathcal{O}_{K}$, $a_n(x)$ and $b_n(x)$ are defined by the formulas (\ref{def-an}) and (\ref{def-bn}).  In order to determine the linearization of $G_{n+1}$, we only need to
get the values modulo $\pi$ of $a_n(x)$ and $b_n(x)$.  In fact, the coefficient $a_{n}(x)~(\!\!\!\!\mod \pi)$ is always constant on $x_{i}+\pi^{n}\mathcal{O}_{K}$, and  the coefficient  $b_{n}(x)$ $(\!\!\!\!\mod \pi)$ is also constant on $x_{i}+\pi^{n}\mathcal{O}_{K}$ but under the condition $a_{n}(x)\equiv 1$ $\ ({\rm mod} \ \pi)$ (see \cite{Liao}). For simplicity, we sometimes write $a_{n}$ and $b_{n}$ without mentioning $x$ if there will be no confusion.

We remind that the cardinality of $\mathbb{K}=\mathcal{O}_{K}/\pi\mathcal{O}_{K}$ is $p^{f}$. The characteristic
of the field $\mathbb{K}$ is $p$. The multiplicative group $\mathbb{K}^{*}$ is cyclic of order $p^{f}-1$.

In the following we shall study the behavior of the finite dynamics $F_{n+1}$ on the $F_{n+1}$-invariant set $X$ and determine all
cycles in $X$ of $F_{n+1}$, which will be called lifts of $\sigma$. For details see \cite{Liao}.

We distinguish the following four behaviors of $F_{n+1}$ on $X$:\\
\indent(a) If $a_n\equiv 1 ~(\!\!\!\!\mod \pi)$ and $b_n\not\equiv 0~(\!\!\!\!\mod \pi)$, then $\Phi$ preserves $p^{f-1}$ cycles of length $p$, so that $F_{n+1}$ restricted
to $X$ preserves $p^{f-1}$ cycles of length $pk$. In this case we say $\sigma$ \emph{grows}.\\
\indent(b) If $a_n\equiv 1 ~(\!\!\!\!\mod \pi)$ and $b_n\equiv 0~(\!\!\!\!\mod \pi)$, then $\Phi$ is the identity, so  that $F_{n+1}$ restricted to $X$ preserves $p^{f}$ cycles of length $k$. In this case we say $\sigma$ \emph{splits}.\\
\indent(c) If $a_n\equiv 0~(\!\!\!\!\mod \pi )$, then $\Phi$ is constant, so that  $F_{n+1}$ restricted to $X$ preserves one cycle of length $k$ and the remaining points of $X$ are mapped into this cycle. In this case we say $\sigma$ \emph{grow tails}.\\
\indent(d) If $a_n\not\equiv 0,1~(\!\!\!\!\mod \pi)$, then $\Phi$ is a permutation and the $l$-th iterate of $\Phi$ reads
$$\Phi^{l}(x,t)=b_n(a_n^{l}-1)/(a_n-1)+a^{l}_{n}t$$
so that
$$\Phi^{l}(x,t)-t=(a_n^{l}-1)(t+\frac{b_{n}}{a_n-1}).$$
Thus, $\Phi$ admits a single fixed point $t=-b_{n}/(a_n-1)$, and the remaining points lie on cycles of length $\ell$, where $\ell$ is the order of $a_n$ in $\mathbb{K}^{*}$. So, $F_{n+1}$ restricted to $X$ preserves one cycle of length $k$ and $\frac{p^{f}-1}{\ell}$ cycles of length $k\ell$. In this case we say $\sigma$ \emph{partially splits}.

For $n\geq 1$, let $\sigma=(x_1,\dots,x_k)\subset \O/\pi^{n}\O$ be a $k$-cycle and
let $\tilde{\sigma}$ be a lift of $\sigma$.
Now we shall show the relation between $(a_n, b_n)$ and $(a_{n+1},
b_{n+1})$. Our aim is to find the change of nature from a cycle to
its lifts.

\begin{lemma}\label{anbn}
Let $\sigma=(x_1,\dots,x_k)$ be a $k$-cycle of $F_n$ and let
$\tilde{\sigma}$ be a lift of $\sigma$ of length $kr$, where $r\geq
1$ is an integer.  We have
\begin{eqnarray}\label{an}
a_{n+1}(x_i+\pi^nt) \equiv a_n^r(x_i) \quad ({\rm mod} \ \pi^{n}),
\qquad (1\leq i \leq k, t\in C)
\end{eqnarray}
\begin{equation}\label{bn}
  \begin{split}
 & \pi b_{n+1}(x_i+\pi^n t)\\
\equiv& t(a_n(x_i)^r-1)  + b_n(x_i)(1+a_n(x_i)+ \cdots +
a_n(x_i)^{r-1})  \quad ({\rm mod} \ \pi^{n}).
\end{split}
\end{equation}
\end{lemma}

\begin{proof}
 The formula (\ref{an}) follows from
\[
a_{n+1} \equiv (F^r)'(x_i+\pi^n t) \equiv (F^r)'(x_i) \equiv
\prod_{j=0}^{r-1} F'(F^j(x_i))
 \equiv a_n^r \ ({\rm mod} \ \pi^{n}).
 \]

 By repeating $r$-times of the linearization (\ref{linearization}), we obtain
 $$
 F^r(x_i+\pi^nt) \equiv x_i + \Phi^r(x_i, t) \pi^n \quad ({\rm mod} \
 \pi^{2n}),
 $$
 where $\Phi^r$ means the $r$-th iteration of $\Phi$ as function
 of $t$. However,
$$
\Phi^r(x_i ,t) =  ta_n(x_i)^r + b_n(x_i)
 (1+a_n(x_i)+ \cdots + a_n(x_i)^{r-1}).
$$
Thus (\ref{bn}) follows from
  the definition of $b_{n+1}$ and the above two expressions.
\end{proof}
By Lemma \ref{anbn}, we have the following
proposition.
\begin{proposition}
Let $n\geq 1$. Let $\sigma$ be a $k$-cycle of $F_n$
and $\tilde{\sigma}$ be a lift of $\sigma$. Then we have\\
 \indent {\rm 1)} if $a_n \equiv 1 \ ({\rm mod} \ \pi)$, then $a_{n+1} \equiv 1 \ ({\rm mod} \ \pi)$;\\
 \indent {\rm 2)} if $a_n \equiv 0 \ ({\rm mod} \ \pi)$, then $a_{n+1} \equiv 0 \ ({\rm mod} \ \pi)$;\\
 \indent {\rm 3)} if $a_n \not\equiv 0,1 \ ({\rm mod} \ \pi)$ and $\tilde{\sigma}$ is of length
 $k$, then $a_{n+1} \not\equiv 0,1 \ ({\rm mod} \ \pi)$;\\
 \indent {\rm 4)} if $a_n \not\equiv 0,1 \ ({\rm mod} \ \pi)$ and $\tilde{\sigma}$ is of length
 $kd$ where $d\ge 2$ is the order of $a_n$ in $(\O/\pi\O)^*$, then $a_{n+1} \equiv 1 \ ({\rm mod} \ \pi)$.
\end{proposition}
By the above analysis, the case of growing tails is simple.
If $\sigma=(x_1, \cdots, x_k)$ is a cycle of $F_n$ which grows
tails, then $F$ admits a $k$-periodic point $x_0$ in the clopen set
$\mathbb{X} =\bigsqcup_{i=1}^k (x_i +\pi^n \O)$ and $\mathbb{X}$ is
contained in the attracting basin of the periodic orbit $x_0,
F(x_0), \cdots, F^{k-1}(x_0)$.

Similarly, for the case of partially splitting, we can also find a periodic orbit in $\mathbb{X}$, and other parts are reduced to the cases of growing and splitting.

We say a cycle splits $\ell$ times if
itself splits, its lifts split, the second generation of the
descendants split, ... , and the ($\ell-1$)-th generation of the
descendants split.

Let $\sigma=(x_1,x_2,\cdots,x_k)$ be a $k$-cycle of $F_n$. For a given sequence of positive integers $\vec{E}=(E_j)_{j\geq 1}$. The cycle $\sigma$ is said to be of \emph{type }$(k,\vec{E})$ if it is a growing $k$-cycle at level $n$ and the lift of this $k$-cycle split $E_1-1$ times; all the $E_1$-th
generation of descendants grow and then all the lifts split $E_2-1$ times; again, the descendants grow and the lifts of the
 growing descendants split $E_3-1$ times, and so on.
  The $F$-invariant
clopen set $\mathbb{X}=\bigsqcup_{i=1}^{k}(x_{i}+\pi^n\mathcal{O}_{K})$
and the system $(\mathbb{X},F)$ are then said to be of type $(k,\vec{E})$ at level $n$.  We say that $\mathbb{X}$
is of \emph{type }$(k,e)$ at level $n$ if $\mathbb{X}$ is of $(k,\vec{E})$ at level $n$ with $\vec{E}=(e,e,\cdots)$, where $e$ is the ramification index of $K$ over $\Qp$.

\subsection{Dynamics of multiplications on a finite extension of $\Qp$}
For an affine polynomial $F(x)=\alpha x+\beta$ with $\a,\b\in K$, we consider the dynamics $(K,F)$ by distinguishing four cases:\\
\indent(1) if $\a=1, \b=0$, then $F$ is an identity map on $K$;\\
\indent(2) if $\a=0$, then $F$ is a constant map;\\
\indent(3) if $\a=1, \b\not=0$, then  $F$ is conjugate to the translation $\hat{F}(x)=x+1$ by $h(x)=x/\b$;\\
\indent(4) if $\a\neq 0, 1$,  then $F$ is conjugate to the multiplication
$\hat{F}(x) = \a x$  by $h(x) = x-(\b/(1-\a))$.

 The dynamical systems $(K,F)$ of the first two cases are trivial. For the translation $F(x)=x+1$, each closed disk of radius $1$ is $F$-invariant and each subsystem $(\CD(a,1),F)$ is conjugate to $(\O,F)$ by $h(x)=x-a$.
The dynamical system $(\O,F)$ is of type $(1,e)$ at level $0$, for details see \cite{Liao}.

Now we are going to study the multiplication dynamics $(K,F(x)=\a x)$.
It is easy to see that $0,\infty$ are  two fixed points of $F$.
We distinguish two cases.
%

Case (A) $\a\not \in \U$. If $|\a|_p<1$, then the system admits an attracting fixed point $0$ with the whole $K$ being the attracting basin, that is,
$$\lim_{n\rightarrow \infty } F^{n}(x) = 0,~~~~~~~~~\forall x\in K .$$
 If $|\a|_p>1$, then the system admits a repelling fixed point $0$ and the whole $K$ except $0$ lie in  the attracting basin of $\infty$, that is,
$$\lim_{n\rightarrow \infty } F^{n}(x) = \infty,~~~~~~~~~\forall x\in K \setminus \{0\}.$$

Case (B) $\a\in  \U$.  For each $n\in \Z$, the sphere $\SP(0,p^{n/e})=\pi^{-n}\U$  is $F$-invariant.
We decompose $K$ as
$$K=\{0\}\bigcup (\bigcup_{n=-\infty}^{\infty}\pi^{n}\U).$$
The dynamical system $(\pi^{-n}\U,F)$ is conjugate to the system $(\U,F)$
by $h(x)=\pi^{n}x$. See the following commuting graph.
\begin{center}
\begin{picture}(100,90)(0,-65) \put(-5,0){$\pi^{-n}\U$}
\put(25,6){\vector(1,0){40}} \put(70,0){$\pi^{-n}\U$} \put(37,10){$\alpha x
$} \put(10,-10){\vector(0,-1){30}} \put(-10,-25){$\pi^{n} x$}
\put(79,-10){\vector(0,-1){30}} \put(82,-25){$\pi^{n} x $}
\put(8,-55){$\U$} \put(75,-55){$\U$}
\put(25,-52){\vector(1,0){40}} \put(37,-63){$\alpha x$}
\end{picture}
\end{center}
So we are going to study the dynamics $(\U,F)$. We distinguish the  following two cases.\\
\indent(i) $\a\in  \V \setminus \{1\}$. Let $\ell$ be the order of $\a$, i.e., the least integer such that $\a^{\ell}=1 $. Then every point is in a periodic orbit with period $\ell$.\\
\indent(ii) $\a\in  \U \setminus \V$. We obtain the following results which will be useful for our studying of the homographic dynamics.

\begin{proposition}\label{affine}
Let $\alpha\in \mathbb{U}\setminus \mathbb{V}$.  Consider the dynamical system $(\mathbb{U},F)$, where $F(x)=\alpha x$. Let $\ell$ be the order of $\alpha$ in $\mathbb{K}^{*}$. One can decompose $\mathbb{U}$ into
   $(p^f-1) p^{v_{\pi}(\a^\ell-1)\cdot f-f}/\ell$ clopen sets such that each
  clopen set is of type $
(\ell, \vec{E})$ at level $v_{\pi}(\a^\ell-1)$, where
$$\vec{E}=\left(\val(\frac{\a^{\ell p }-1}{\a^{\ell}-1}), \ \val(\frac{\a^{\ell p^2 }-1}{\a^{\ell p}-1}), \ \cdots, \ \val(\frac{\a^{\ell p^{N+1} }-1}{\a^{\ell p^N}-1}), \ e, \ e, \ \cdots\right).$$
Here $N$ is the largest integer such that $\val((\a^{\ell p^{N+1} }-1)/({\a^{\ell p^N}-1}))\neq  e$.

Moreover, if the ramification index $e=1\mbox{ or } 2$, then
 $$\val(\frac{\a^{\ell p^{i+1}}-1}{\a^{\ell p^{i}}-1})=e$$
 for all positive integers $i\geq 2$.
\end{proposition}
\begin{proof}
Notice that  $\ell $ is  the order of $\a$ in $\K^* $. Thus $\a^{\ell} \equiv 1 \ ({\rm mod} \ \pi)$.
  We can check that there are  $(p^f-1)/\ell $ cycles of length $\ell
  $ at level $1$.
Now we consider the $\ell$-cycles at level $1$.

For each $x_0\in \U$,
  \begin{eqnarray*}
a_1(x_0)&=&(f^\ell)'(x_0)= \a^\ell, \\
b_1(x_0)&=&\frac{f^\ell(x_0)-x_0}{\pi}=\frac{(\a^\ell-1)x_0}{\pi}.
  \end{eqnarray*}

Thus $a_1(x_0)\equiv 1 \ ({\rm mod} \ \pi)$ and
  $\val(b_{1}(x_0))=\val(\a^\ell-1)-1.$
By induction, one can check that each $\ell$-cycle at level $1$ splits $\val(\a^\ell-1)-1 $ times and then
all its lifts grow.

Let $m=\val(\a^\ell-1)+1 $. Consider the $\ell p$-cycles at level
$m$.
 For each $x_0\in \U$,
  \begin{eqnarray*}
a_m(x_0)&=&(f^{\ell p})'(x_0)= \a^{\ell p}, \\
b_m(x_0)&=&\frac{f^{\ell p}(x_0)-x_0}{\pi^m}=\frac{(\a^{\ell
p}-1)x_0}{\pi^m}.
  \end{eqnarray*}
Then we have
\[
 \val(b_m(x_0))=\val({\a^{\ell p}-1})-m=\val(\frac{\a^{\ell p }-1}{\a^{\ell}-1})-1.
\]
Hence each $\ell p$-cycle splits at level $m$ split $\val(\frac{\a^{\ell p }-1}{\a^{\ell}-1})-1$ times and then all its lifts
grow.

Now let $q=\val(\a^{\ell p}-1)+1 $. Consider the $\ell p^2$-cycles at level
$q$. By the same calculations, we have for each $x_0$,
$a_q(x_0)= \a^{\ell p^2}\equiv 1 \ ({\rm mod} \ \pi)$ and
\[
 \val(b_q(x_0))=\val({\a^{\ell p^2}-1})-q=\val(\frac{\a^{\ell p^2 }-1}{\a^{\ell p}-1})-1.
\]
Hence each $\ell p^2$-cycle splits $\val(\frac{\a^{\ell p^2 }-1}{\a^{\ell p}-1})-1$ times and then all its lifts
grow.

Go on this process, we can show that each $\ell p^k$-cycle  at level $\val(\a^{\ell p^{k-1}}-1)+1 $ splits $\val(\frac{\a^{\ell p^k }-1}{\a^{\ell p^{k-1}}-1})-1 $ times and then all its lifts
grow.

Since for $1\leq i \leq p-1$, $\val(\a^{\ell i p^k} -1) \to \infty$ when $k\to \infty$  (see \cite{sch}, p.100), and $\val(p)=e$, we can find
an integer $N$ such that for all $k> N$, we have
\begin{align*}
\val(\frac{\a^{\ell p^{k+1} }-1}{\a^{\ell p^k}-1})-1&=\val(1+\a^{\ell p^k}+\cdots+\a^{(p-1)\ell p^{k}})-1\\
&=\val((1-1)+(\a^{\ell p^k}-1)+\cdots+(\a^{(p-1)\ell p^{k}}-1)+p)-1\\
&=\val(p)-1=e-1.
\end{align*}
Therefore, we can conclude that each $\ell$-cycle at level $\val(\a^\ell-1)$ is of type $
(\ell, \vec{E})$ with
$$\vec{E}=\left(\val(\frac{\a^{\ell p }-1}{\a^{\ell}-1}), \ \val(\frac{\a^{\ell p^2 }-1}{\a^{\ell p}-1}), \ \cdots, \ \val(\frac{\a^{\ell p^{N+1} }-1}{\a^{\ell p^N}-1}), \ e, \ e, \ \cdots\right).$$

On the other hand, we have the number of $\ell$-cycle at level $\val(\a^\ell-1)$ is
$$\frac{p^f-1}{\ell} \cdot (p^{f})^{\val(\a^\ell-1)-1},$$ since one $\ell$-cycle splits into $p^f$ number of $\ell$-cycles.

When the ramification index $e$ is $1$ or $2$, it is easy to check that
$\val(\a^{\ell p^i}-1)> e$ for all $i\geq 2$. Let $s=\val(\a^{\ell p^i}-1)$. Write $\a^{\ell p^{i}}=1+\pi^{s}\eta$ for some $\eta \in \U$.
 Then we have $$\a^{\ell p^{i+1}}=(1+\pi^{s}\eta)^{p}=1+{p\choose 1}\pi^{s}\eta+{p\choose 2}(\pi^{s}\eta)^{2}+\cdots+{p\choose p}(\pi^{s}\eta)^{p}.$$
 We need to compare the $\val$-values of the last term and the second term in the above summation. Since $s>\frac{e}{p-1}$, we have
  $$\a^{\ell p^{i+1}}=1+p\pi^{s}\eta ~~(\!\!\!\!\mod  \pi^{s+e}).$$
Then $\val(\frac{\a^{\ell p^{i+1}}-1}{\a^{\ell p^{i}}-1})=e$.

\end{proof}
\begin{remark}\label{remark}
In the proof of last assertion of Proposition \ref{affine}, the essential point is $s>\frac{e}{p-1}$. So one can easily check that
$$\val(\frac{\a^{\ell p^{2}}-1}{\a^{\ell p}-1})=e,$$ except for the case where $e=2, p=2$ and $s=\val(\a^{\ell p}-1)=2$.
\end{remark}
\begin{corollary}\label{affinep>3}
Let $\alpha\in \mathbb{U}\setminus \mathbb{V}$.  Consider the dynamical system $(\mathbb{U},\alpha x)$.  Assume $p\geq 3$. If $\val(\a^\ell-1)>\frac{e}{p-1}$, then $\mathbb{U}$ is decomposed into $(p^f-1) p^{v_{\pi}(a^\ell-1)\cdot f-f}/\ell$ clopen sets such that each
clopen set is of type $(\ell,e)$ at level $v_{\pi}(a^\ell-1)$.
\end{corollary}
\begin{proof}Let $m=\val(\a^\ell-1)$. Consider the $\ell p$-cycles at level $m+1$, we have
\[\val(b_{m+1})=\val(\a^{\ell p}-1)-\val(\a^\ell-1)-1.\]
 Write $\a^{\ell}=1+\pi^{m}\eta$ for some $\eta \in \U$.
 Then we have $$\a^{\ell p}=(1+\pi^{m}\eta)^{p}=1+{p\choose 1}\pi^{m}\eta+{p\choose 2}(\pi^{m}\eta)^{2}+\cdots+{p\choose p}(\pi^{m}\eta)^{p}.$$
 We need to compare the $\val$-values of the last term and the second term in the above summation.
 Since $m>\frac{e}{p-1}$, we have $mp>e+m$. Then
 $$\a^{\ell p}=1+p\pi^{m}\eta ~~\ ({\rm mod} \  \pi^{e+m+1}).$$
So $$\val(b_{m+1}(x_i+\pi^m t))=e-1,$$ which implies that the $\ell p$-cycles at level $\val(\a^\ell-1)+1$ splits $e-1$ times and then all the lifts grows.
By induction, each $\ell$-cycle at level $\val(\a^\ell-1)$  is of type $(\ell,e)$.
\end{proof}

\section{Minimal decomposition when $\phi$ admits no fixed point}\label{detail}

In this main part of the paper, we determine the minimal decomposition of a homographic map $\phi$ without fixed point in $\Q_p$.

Assume   $\sqrt{\Delta}\notin \Qp$. Then $\Qp$
contains no fixed point of $\phi$. Consider the quadratic extension
$\Qp(\sqrt{\Delta})$ of $\Qp$. Then in $\Qp(\sqrt{\Delta})$, $\phi$ has two fixed points
\[x_{1}=\frac{a-d+\sqrt{\Delta}}{2c},\quad x_{2}=\frac{a-d-\sqrt{\Delta}}{2c}.\]
 From now on, let $K:=\Qp(\sqrt{\Delta})$ denote the quadratic extension of $\Q_p$ generated by $\sqrt{\Delta}$
 and let $\mathbb{K}$ be the residual class field of $K$.
 Since $K$ is a quadratic extension of $\Qp$,
$K$ is either a totally ramified extension or an unramified extension.
Let $$ r_0=|x_1-x_2|_{p}=\left|\frac{\sqrt{\Delta}}{c}\right|_p \quad \text{and}\quad \lambda
=\frac{a+d+\sqrt{\Delta}}{a+d-\sqrt{\Delta}} \in K\setminus \Qp.$$
The following proposition  shows that $\phi$ is always conjugate to a multiplication.
We can easily prove it by checking the claimed conjugacy.
\begin{proposition}
The dynamical system $(\PK, \phi)$ is topologically conjugate to $(\PK,
\psi)$ where $\psi$ is the multiplication
$
\psi(x)= \lambda x.
$
In other words, $g\circ \phi \circ g^{-1} = \psi$ where
$$
    g(x) = \frac{x-x_2}{x-x_1}.
$$
\end{proposition}
If $|a+d|_{p}\neq |\sqrt{\Delta}|_{p}$,   it is easy to see that $|a+d+\sqrt{\Delta}|_{p}=|a+d-\sqrt{\Delta}|_{p}$.
 If $|a+d|_{p}= |\sqrt{\Delta}|_{p}$, we have the same result as the following lemma shows.
\begin{lemma}\label{lambda}
If $|a+d|_{p}= |\sqrt{\Delta}|_{p}$, then we have $|a+d+\sqrt{\Delta}|_{p}=|a+d-\sqrt{\Delta}|_{p}$.
\end{lemma}
\begin{proof}
 For $p\geq 3$, Corollary \ref{sqrtx+yp>2} implies immediately that $$|a+d+\sqrt{\Delta}|_{p}=|a+d-\sqrt{\Delta}|_{p}=|\sqrt{\Delta}|_{p}.$$
In the following we discuss the case  $p=2$. If $\Q_2(\sqrt{\Delta})=\Q_2(\sqrt{i})$ with $i=2,-2,6$ or $-6$, then $|a+d|_{p}\neq |\sqrt{\Delta}|_{p}$.
So we need only consider the following two cases. \\
 \indent(1) If $\Q_2(\sqrt{\Delta})=\Q_2(\sqrt{-3})$, the assertion (1) of Proposition \ref{x-sqrt} implies that
$$|a+d+\sqrt{\Delta}|_{2}=|a+d-\sqrt{\Delta}|_{2}=\frac{1}{2}|\sqrt{\Delta}|_{2}.$$
\indent(2) If $\Q_2(\sqrt{\Delta})=\Q_2(\sqrt{-1})$ or $\Q_2(\sqrt{3})$, the assertion (2) of Proposition \ref{x-sqrt} leads to
$$|a+d+\sqrt{\Delta}|_{2}=|a+d-\sqrt{\Delta}|_{2}=\frac{\sqrt{2}}{2}|\sqrt{\Delta}|_{2}.$$
\end{proof}

A direct consequence of Lemma \ref{lambda}  is $|\lambda|_p=1$.
From the point of dynamics, we can also explain this. Since $\P$ is $\phi$-invariant, it is easy to see that  neither $x_{1}$ nor  $x_{2}$ is attracting. So $|\lambda|_{p}=1$.  We will distinguish two cases. \\
\indent(a) $\lambda^n=1$ for some $n\ge 1$. Let $\ell$ be the order of
$\lambda$, i.e.,the least positive integer such that
$\lambda^{\ell}=1$.
  It is easy to see that all points are in a periodic orbit with period $\ell$.\\
\indent(b) $\lambda^n\neq 1 $ for all $n\ge 1$. We will treat the cases $p\geq 3$ and $p=2$ separately.

In the remainder of this section we assume that $\lambda=\frac{a+d+\sqrt{\Delta}}{a+d-\sqrt{\Delta}}$ is not a root of unity and let $\ell$ be the order in the group $\mathbb{K}^*$ of
$\lambda$ and $\pi$ be a uniformizer of $K$.

For $p\geq 3$, we have the following two theorems corresponding to the unramified quadratic extension and the ramified quadratic  extensions.
\begin{theorem}\label{mindecunrami}
Assume $p\geq 3$ and $\sqrt{\Delta}\notin \Qp$. Suppose that
$K=\Qp(\sqrt{\Delta})$ is an unramified quadratic extension of $\Qp$.  Then
$\ell|(p+1)$
 and the dynamics $(\mathbb{P}^{1}(\Qp), \phi)$ is decomposed into
$((p+1)p^{v_{p}(\lambda^{\ell}-1)-1})/\ell$ minimal subsystems.
The minimal subsystems are topologically conjugate to the
adding machine on the odometer $\mathbb{Z}_{(p_s)}$  with $(p_s)=(\ell,\ell p,\ell p^2,\cdots)$.
\end{theorem}

\begin{theorem}\label{rami-decomposition-p>2}
Assume $p\geq 3$ and $\sqrt{\Delta}\notin \Qp$. Suppose that $K=\Qp(\sqrt{\Delta})$ is
a ramified quadratic extension of $\Qp$. Since $|a+d|_{p}\neq  |\sqrt{\Delta}|_{p}$, we distinguish two cases. \\
\indent \mbox{\rm (1)} \ If $|a+d|_{p}> |\sqrt{\Delta}|_{p}$, then
$\lambda= 1 ~(\!\!\!\mod ~\pi)$, and the dynamics
$(\mathbb{P}^{1}(\Qp), \phi)$ is decomposed into
$2p^{(v_{\pi}(\lambda^{p}-1)-3)/2}$ minimal subsystems, such that each minimal subsystem is conjugate to
the adding machine on the odometer $\Z_{(p_{s})}$ with $(p_{s})=(1,p,p^{2},\cdots).$  \\
 \indent \mbox{\rm (2)} \ If $|a+d|_{p}< |\sqrt{\Delta}|_{p}$, then
$\lambda=- 1 ~(\!\!\!\mod ~\pi)$, and the dynamics
$(\mathbb{P}^{1}(\Qp), \phi)$ is decomposed into
$p^{(v_{\pi}(\lambda^p+1)-3)/2}$ minimal subsystems, such that each minimal subsystem is  conjugate to
the adding machine on the odometer $\Z_{(p_{s})}$ with $(p_{s})=(2,2p,2p^{2},\cdots).$
\end{theorem}

As  consequences of Theorems \ref{mindecunrami} and \ref{rami-decomposition-p>2}, we find conditions under which the whole space is minimal and describe the corresponding dynamical structure.
\begin{corollary}
The system $(\P,\phi)$ is minimal if and only if one of the following conditions is satisfied:
\begin{itemize}
  \item[\rm{(1)}] $\Qp(\sqrt{\Delta})$ is an unramified quadratic extension
of $\Qp$, $\ell=p+1$ and $v_{p}(\lambda^{\ell}-1)=1$;
  \item[\rm{(2)}] $\Qp(\sqrt{\Delta})$ is a ramified  quadratic extension of $\Qp$ and $\val(\lambda^{p}+1)=3$.
\end{itemize}
\end{corollary}

\begin{corollary}
If the  system $(\P,\phi)$ is minimal, then $(\P,\phi)$ is topologically conjugate to the adding machine on  the odometer $\mathbb{Z}_{(p_s)}$  with
$$(p_s)=(2,2p, 2p^2,\cdots) \quad \mbox{~~~or~~~} \quad (p_s)=\big(p+1,(p+1)p, (p+1)p^2,\cdots\big).$$
\end{corollary}

When $p=2$, there are three different situations.
\begin{theorem}\label{decomp-p=2-unrami}
Assume $p=2$ and $K=\Q_2(\sqrt{-3})$.  Then the dynamical system $(\mathbb{P}^{1}(\Q_2),\phi)$ is decomposed into  $3\cdot2^{v_{2}(\lambda^{2\ell}-1)-2 }/\ell$ minimal subsystems. Moreover, each minimal system is conjugate to  the
adding machine on the odometer $\mathbb{Z}_{(p_s)}$  with $(p_s)=(\ell,\ell 2,\ell 2^2,\cdots)$.
\end{theorem}

\begin{theorem}\label{decompositionsqrt2}
Assume $p=2$ and $K=\Q_2(\sqrt{2}),\Q_2(\sqrt{-2}),\Q_2(\sqrt{6})$ or $\Q_2(\sqrt{-6})$.
 Since  $|\sqrt{\Delta}|_{2}\neq |a+d|_{2}$,
 we distinguish the following two cases.\\
\indent {\rm (1)} If $|a+d|_{2}> |\sqrt{\Delta}|_{2}$, then the dynamical system $(\mathbb{P}^{1}(\Q_{2}),\phi)$ is decomposed into $2^{(v_{\pi}(\lambda-1)-1)/2}$ minimal subsystems. \\
\indent {\rm (2)} If $|a+d|_{2}< |\sqrt{\Delta}|_{2}$, then the dynamical system $(\mathbb{P}^{1}(\Q_{2}),\phi)$ is decomposed into $2^{(v_{\pi}(\lambda+1)-1)/2}$ minimal subsystems.\\
In both cases, each minimal system is conjugate to  the
adding machine on the odometer $\mathbb{Z}_{(p_s)}$  with $(p_s)=(1, 2, 2^2,\cdots)$.
\end{theorem}

\begin{theorem}\label{decompositionsqrt3}
Assume $p=2$ and $K=\Q_2(\sqrt{-1}),$~or~$\Q_2(\sqrt{3})$.  We distinguish the following three cases.\\
\indent {\rm (1)} If $|a+d|_2=|\sqrt{\Delta}|_{2}$, the system $(\mathbb{P}^{1}(\Q_{2}) ,\phi)$ is decomposed into $2^{(v_{\pi}(\lambda^{2}+1)-2)/2}$ minimal subsystems.\\
\indent {\rm (2)} If $|a+d|_2>|\sqrt{\Delta}|_{2}$, the system $(\mathbb{P}^{1}(\Q_{2}),\phi)$ is decomposed into $2^{v_{\pi}(\lambda-1)/2}$ minimal subsystems.\\
\indent {\rm (3)} If $|a+d|_2<|\sqrt{\Delta}|_{2}$, the system $(\mathbb{P}^{1}(\Q_{2}),\phi)$ is decomposed into $2^{v_{\pi}(\lambda+1)/2}$ minimal subsystems.\\
In all the three cases, each minimal system
is conjugate to  the
adding machine on the odometer $\mathbb{Z}_{(p_s)}$  with $(p_s)=(1, 2, 2^2,\cdots)$.
\end{theorem}
As a consequence of Theorems \ref{decomp-p=2-unrami}, \ref{decompositionsqrt2} and \ref{decompositionsqrt3}, we obtain necessary and sufficient conditions
under which the dynamics $(\mathbb{P}^{1}(\Q_{2}),\phi)$ is minimal.
\begin{corollary}
The system $(\mathbb{P}^{1}(\Q_{2}),\phi)$ is minimal if and only if one of the following conditions is satisfied:
\begin{itemize}
  \item[\rm{(1)}] $K=\Q_2(\sqrt{-3})$, $\ell=3$ and $v_{2}(\lambda^{2\ell}-1)=2$;
  \item[\rm{(2)}] $K=\Q_2(\sqrt{-1})~
\mbox{ or }\Q_2(\sqrt{3})$,  $|a+d|_{2}=|\sqrt{\Delta}|_{2}$ and $v_{\pi}(\lambda^{2}+1)=2$.
\end{itemize}
\end{corollary}

In the following, we shall prove Theorems \ref{mindecunrami}-\ref{rami-decomposition-p>2} and Theorems \ref{decomp-p=2-unrami}-\ref{decompositionsqrt3}.
But before doing that, we will first compute the distances from the fixed points of $\phi$ to the
field $\mathbb{Q}_{p}$.

\subsection{Distances from the fixed points to $\Q_p$}
In the proofs of  Theorems \ref{mindecunrami}-\ref{rami-decomposition-p>2} and Theorems \ref{decomp-p=2-unrami}-\ref{decompositionsqrt3}, we need to know the distances from the fixed points to $\Q_p$.  What we have proved in Section \ref{distances} will be useful.

Recall that
\[x_{1}=\frac{a-d+\sqrt{\Delta}}{2c},\quad x_{2}=\frac{a-d-\sqrt{\Delta}}{2c},\]
and
\[ r_0=|x_1-x_2|_{p}=\left|\frac{\sqrt{\Delta}}{c}\right|_p. \]

\begin{lemma}\label{distance}
 Let $p\geq 3$. Then  $d(x_1,\Qp)= d(x_2,\Qp)
=r_{0}$.
\end{lemma}
\begin{proof}
Since $x_1 = \frac{a-d +\sqrt{\Delta}}{2c}$ with $\frac{a-d}{2c}\in
\Qp$, we have
$$d(x_1,\Qp)=\inf_{x\in \Qp}|x-x_1|_{p}=\inf_{x\in \Qp}|x-\sqrt{\Delta}/(2c)|=d(\sqrt{\Delta}/(2c),\Qp).$$
By Corollary \ref{p>2distance}, we get
$$
d(\Qp,x_1)=|\sqrt{\Delta}/(2c)|_{p}. $$
The same arguments apply to $x_2$ instead of $x_1$.

Notice that $|\sqrt{\Delta}/(2c)|_{p}=|\sqrt{\Delta}/c|_{p}$ when
$p\geq 3$.
So  $d(x_1,\Qp)= d(x_2,\Qp)=r_{0}$.
\end{proof}

\begin{lemma}\label{subsystem}
Let $p\geq 3$. We have\\
\indent {\rm (1)} \ \ $\Qp\subset K\setminus
(\mathbb{D}(x_1,r_0)\cup
\mathbb{D}(x_2,r_0))$;\\
\indent {\rm (2)} \ \ $g(\mathbb{P}^{1}(\Qp))\subset
\mathbb{S}(0,1)$.
\end{lemma}
\begin{proof} The first assertion is a direct consequence of Lemma \ref{distance}.\\
For the second assertion, by Propositon \ref{disk}, we have
$$g(\mathbb{D}(x_1,r_0))= \mathbb{P}^{1}\setminus \overline{\mathbb{D}}(0,1) \mbox{~and~}g(\mathbb{D}(x_2,r_0))=\mathbb{D}(0,1).$$
So $g(\mathbb{P}^{1}(\Qp))\subset
\mathbb{S}(0,1)$.
\end{proof}

\begin{lemma}\label{distancep=2Lemma}
\begin{itemize}
  \item[\rm{(1)}] If $\Q_2(\sqrt{\Delta})=\Q_2(\sqrt{-3})$, then  $d(x_1,\Q_{2})= d(x_2,\Q_{2})=r_0$.
  \item[\rm{(2)}] If $\Q_2(\sqrt{\Delta})=\Q_2(\sqrt{2}),\Q_2(\sqrt{-2}),\Q_2(\sqrt{-6})$ or $\Q_2(\sqrt{6})$, then $$d(x_1,\Q_{2})= d(x_2,\Q_{2})=2r_0.$$
  \item[\rm{(3)}] If $\Q_2(\sqrt{\Delta})=\Q_2(\sqrt{-1})$~or~$\Q_2(\sqrt{3})$, then $d(x_1,\Q_{2})= d(x_2,\Q_{2})=\sqrt{2}r_0$.
\end{itemize}
\end{lemma}
\begin{proof}
Since $x_1 = \frac{a-d +\sqrt{\Delta}}{2c}$ with $\frac{a-d}{2c}\in
\Q_{2}$, we have
$$d(x_1,\Q_{2})=\inf_{x\in \Q_{2}}|x-x_1|_{2}=\inf_{x\in \Q_{2}}|x-\sqrt{\Delta}/(2c)|=d(\sqrt{\Delta}/(2c),\Q_{2}).$$
\indent(1)  Assume $\Q_2(\sqrt{\Delta})=\Q_2(\sqrt{-3})$. By Corollary \ref{distancep=2}, we have
$$d(\sqrt{\Delta}/(2c),\Q_{2})=\frac{|\sqrt{\Delta}/(2c)|_{2}}{2}.$$
So $$d(x_1,\Q_{2})=\frac{|\sqrt{\Delta}/(2c)|_{2}}{2}=|\sqrt{\Delta}/c |_{2}=r_0.$$ \\
\indent(2) Assume $\Q_2(\sqrt{\Delta})=\Q_2(\sqrt{2}),\Q_2(\sqrt{-2}),\Q_2(\sqrt{-6})$ or $\Q_2(\sqrt{6})$.  By Corollary \ref{distancep=2}, we have
$$d(\sqrt{\Delta}/(2c),\Q_{2})=|\sqrt{\Delta}/(2c)|_{2}.$$
So
 $$d(x_1,\Q_{2})=\left|\frac{\sqrt{\Delta}}{2c}\right|_{2}=2r_{0}.$$
\\
\indent(3) Assume $\Q_2(\sqrt{\Delta})=\Q_2(\sqrt{-1})$~or~$\Q_2(\sqrt{3})$.  By Corollary \ref{distancep=2}, we have
$$d(\sqrt{\Delta}/(2c),\Q_{2})=\frac{\sqrt{2}|\sqrt{\Delta}/(2c)|_{2}}{2}.$$
So $$d(x_1,\Q_{2})=\frac{\sqrt{2}}{2}\left|\frac{\sqrt{\Delta}}{2c}\right|_{2}=\sqrt{2}r_{0}.$$
Applying same arguments to $x_2$ instead of $x_1$, we complete the proof.
\end{proof}

\subsection{Proof of Theorem \ref{mindecunrami}}

Assume that $K=\Qp(\sqrt{\Delta})$ is an unramified quadratic  extension of $\Qp$. Then $|K^*|_{p}=|\Qp^*|_{p}$ and $p$ is a uniformizer of $K$.

\begin{lemma}\label{unramify1}
The unit sphere $\mathbb{S}(0,1)$ consists of $p^2-1$ disjoint closed disks with radius $1/p$, and there are
$(p+1)$ such disks which intersect $g(\P)$.
\end{lemma}
\begin{proof}
By Proposition \ref{disk}, we have
  $$g(\PK\setminus \CD(x_1,r_0))=\CD(1,1/p)\subset \SP(0,1),$$
  $$g(\CD(x_1,r_0/p))=\PK\setminus \CD(0,1),$$
$$g(\CD(x_2,r_0/p))= \CD(0,1/p).$$
Observe that $\SP(0,1)$ consists of $p^{2}-1$ closed  disks of radius $1/p$, because $K$ is
an unramified quadratic extension of $\Qp$ and $\CD(0,1)=\SP(0,1)\cup\CD(0,1/p)$. In order to determine
the  number of those
 disks which intersect $g(\P)$, we are going to determine the number of closed
disks $\CD\subset \CD(x_{1},r_{0})$ of radius $r_{0}/p$ such that $\CD\cap\P\neq \emptyset.$
By Lemma \ref{unrami-number}, there are $p$  disks $\CD(a_{i},r_{0}/p)\subset\CD(x_{1},r_{0})$ where  $i=1,2,\cdots, p$, which intersect $\Qp$.
By Lemma \ref{distance}, we have $\CD(x_1,r_{0}/p)\cap \Q_{p}=\emptyset$ and $\CD(x_2,r_{0}/p)\cap \Q_{p}=\emptyset$. So $\CD(a_{i},r_{0}/p)\neq \CD(x_j,r_{0}/p)$ for all $1\leq i \leq p$ and $j=1 \mbox{~or~} 2$.

By Proposition \ref{disk}, each  $g(\CD(a_{i},r_{0}/p))\subset \SP(0,1)$ is a closed disk of radius $1/p$.
Obviously, $g(\CD(a_i,r))\cap g(\P)\neq \emptyset$. Since $g$ is a bijection from $\PK$ into itself,
$\CD(1,1/p), g(\CD(a_1,r_{0}/p)),\cdots,$ $ g(\CD(a_p,r_{0}/p))$ are the $p+1$ closed disks which intersect $g(\P)$.
\end{proof}

\begin{lemma}\label{induction}
For each $n\geq 1$, let $\CD\subset \SP(0,1)$ be a closed disk with radius $1/p^{n}$. If $\CD\cap g(\P)\neq \emptyset$  and if we decompose $\CD$ into $p^{2}$
closed disk with radius $1/p^{n+1}$, then there are $p$ such disks which intersect $g(\P)$.
\end{lemma}
\begin{proof}
We distinguish the following three cases.\\
\indent {\rm (1)} Assume $\CD \cap \CD(1,1/p)= \emptyset$. By Proposition \ref{disk}, $$g^{-1}(\CD)\subset\CD(x_{1},r_{0})\setminus (\CD(x_{1},r_{0}/p)\cup \CD(x_{2},r_{0}/p)),$$
and $g^{-1}(\CD)$ is a closed disk of radius $r_{0}/p^{n}$ and intersects $\Qp$.
  We decompose $g^{-1}(\CD)$ as $p^{2}$ closed disks of radius $r_{0}/p^{n+1}$. By Lemma \ref{unrami-number}, there are $p$
   disks $\CD_{1},\CD_{2},\cdots,\CD_{p}\subset g^{-1}(\CD)$ of radius $r_{0}/p^{n+1}$ such that $\CD_{i}\cap \Qp \neq \emptyset$ for $1\leq i\leq p$. By Proposition \ref{disk} and a simple
 calculation, each $g(\CD_{i})$ is a closed disk with radius $1/p^{n+1}$. Obviously, $$g(\CD_{i})\cap g(\P)\neq \emptyset \quad \mbox{~for~} i=1,\cdots,p.$$
\indent {\rm (2)} Assume $\CD\subset \CD(1,1/p)$ and $1\not\in \CD$. Let $d_{0}=d(1,\CD):= \inf_{x\in \CD}|x-1|_p$ and $m:=-\log_{p}d_{0}$.
 By Proposition \ref{disk}, $g^{-1}(\CD)\subset \SP(x_{1},p^{m}r_{0})$ is a closed disk of radius $p^{2m-n}r_{0}$.
 We conclude by Lemma \ref{unrami-number} and  arguments similar to that used in Case (1).\\
\indent {\rm (3)} Assume $1\in \CD $. By Proposition \ref{disk}, we get
 $$g^{-1}(\CD)=\PK\setminus \CD(x_{1},p^{n-1}r_0).$$
Notice that $$\PK\setminus \CD(x_{1},p^{n-1}r_0)=(\PK\setminus \CD(x_{1},p^{n}r_0)) \cup \SP(x_{1},p^{n}r_{0})$$
  and  that $\SP(x_{1},p^{n}r_{0})$ consists of $p^{2}-1$ closed disks with radius $p^{n-1}r_{0}$. Also $\CD(x_{1},p^{n}r_{0})=\SP(x_{1},p^{n}r_{0})\cup \CD(x_{1},p^{n-1}r_{0})$ and $\CD(x_{1},p^{n-1}r_{0})\cap \Qp\neq \emptyset$. By Lemma \ref{unrami-number}, there are $p-1$ such disks $\CD_{1},\CD_{2},\cdots,\CD_{p-1}\subset \SP(x_{1},p^{n}r_{0})$ of radius $p^{n-1}r_{0}$ which  intersect $\Qp$.
Again by Proposition \ref{disk}, we get $g(\PK\setminus \CD(x_{1},p^{n}r_0))=\CD(1,1/p^{n+1})$.
So $$\CD(1,1/p^{n+1}),~g(\CD_1),~\cdots,g(\CD_{p-1})$$  are the $p$ closed disks of radius $p^{1/(n+1)}$ which intersect  $g(\P)$.
\end{proof}

\begin{lemma}\label{unramify2}
Consider $\psi(x)=\lambda x$. Let $\ell$ be the order of $\lambda$ in $\mathbb{K}^*$. One can decompose $ \mathbb{S}(0,1)$ into $(p^2-1)p^{v_{p}(\lambda^{\ell}-1)\cdot 2-2}/\ell$ clopen sets. Each clopen set is of type $(\ell,1)$ at level $\val(\lambda-1)$.
\end{lemma}
\begin{proof}
Since $e=1$ and $p\geq 3$, we have $e/(p-1)<1$.
By Corollary  \ref{affinep>3}, the conclusion follows.\\
\end{proof}

Having  the above lemmas, we are now ready to determine the minimal decomposition of dynamics $(\P, \phi)$ when $K$ is an  unramified quadratic extension of $\Qp$.

\begin{proof}[Proof of Theorem \ref{mindecunrami}]
The system $(\P,\phi)$ is conjugate to the system $(g(\P),\psi)$ by $g(x) =\frac{x-x_2}{x-x_1}$.
So we are going to study the dynamics $(g(\P),\psi)$.
 By Lemma \ref{subsystem}, $g(\P)\subset \SP(0,1)$. So  $(g(\P),\psi)$ is a subsystem of $(\SP(0,1),\psi)$.
 By Lemma \ref{unramify2}, $\ell |(p^2-1)$. Since $g(\P)$ is $\psi$-invariant, by Lemmas \ref{unramify2} and  \ref{unramify1}, we have $\ell| (p+1)$.

By Lemma \ref{unramify2}, the dynamics $(\SP(0,1),\psi)$ is decomposed into $(p^2-1)p^{v_{p}(\lambda^{\ell}-1)\cdot 2-2}/\ell$ subsystems $(B_i,\psi)$ of type $(\ell,1)$ at level $v_{p}(\lambda^{\ell}-1)$ and each $B_i$ is a $\psi$-invariant clopen set which is a union of $\ell$ closed disks of radius $p^{-v_{p}(\lambda^{\ell}-1)}$. By Lemmas \ref{unramify1} and \ref{induction},
there are $(p+1)p^{v_{p}(\lambda^{\ell}-1)-1}/\ell$ such clopen sets which intersect $g(\P)$.

Let $B_{i}\subset \SP(0,1)$ be an $(\ell,1)$ type clopen set at level $v_{p}(\lambda^{\ell}-1)$ of $\psi$.  If $B_{i}\cap g(\P)\neq \emptyset$,
we claim that the dynamics $(B_{i}\cap g(\P),\psi)$ is minimal.  Since $(B_{i},\psi)$ is of type $(\ell,1)$ at
level $v_{p}(\lambda^{\ell}-1)$, let $(x_{1},\cdots,x_{\ell})$ be the cycle of $\psi_{v_{p}(\lambda^{\ell}-1)}$ at
 level $v_{p}(\lambda^{\ell}-1)$ corresponding to $B_{i}$. Then it follows that the cycle $(x_{1},\cdots,x_{\ell})$  grows. By Lemma \ref{induction}, there exists a unique lift $(y_1,\cdots,y_{\ell p})$ at level $v_{p}(\lambda^{\ell}-1)+1$ such that
$$\CD(y_j,p^{-v_{p}(\lambda^{\ell}-1)-1})\cap g(\P)\neq  \emptyset, ~\mbox{for}~ 1\leq j \leq\ell p.$$
Using Lemma \ref{minimal-part-to-whole}, by induction, we infer that the system $(B_{i}\cap g(\P),\psi)$ is minimal.

Following the above arguments and the proof of Theorem 3 in \cite{FanLiao11}, we also deduce that the minimal subsystem $(B_{i}\cap g(\P),\psi)$ is conjugate to  the
adding machine on the odometer $\mathbb{Z}_{(p_s)}$  with $(p_s)=(\ell,\ell p,\ell p^2,\cdots)$.
\end{proof}

\subsection{Proof of Theorem \ref{rami-decomposition-p>2}}
Assume that $K$ is a ramified quadratic extension of $\Qp$.  Then we  have $e=2$ and $|K^{*}|_{p}=p^{\Z/2}$.
\begin{lemma}\label{ramify1}
The unit sphere $\mathbb{S}(0,1)$ consists of $p-1$ disjoint closed disks of radius $p^{-1/2}$, and there are
$2$ such disks intersecting $g(\P)$.
\end{lemma}
\begin{proof}
Decompose $\mathbb{P}^{1}(K)$ as $\CD(x_1,r_0)$ and $\PK\setminus \CD(x_1,r_0)$. By Proposition \ref{disk}, $$g(\PK\setminus \CD(x_1,r_0))=\CD(1,p^{-1/2}).$$
Observe that $r_{0}=p^{m/2}$ for some odd number $m$, because $\Qp(\sqrt{\Delta})$
is a ramified quadratic extension of $\Qp$ and $r_{0}=\frac{|\sqrt{\Delta}|_{p}}{|c|_{p}}$ with $c,\Delta \in \Qp$.
By Lemma \ref{rami-number}, $\CD(x_{1}, r_{0})$ consists of $p$ closed disk of radius $r_{0}p^{-1/2}$ and there
is a unique  such disk $\CD=\CD(\frac{a-d}{2c},r_{0}p^{-1/2})$ which intersects $\Qp$. By Proposition \ref{disk}, we deduce that $g(\CD)\subset \SP(0,1)$ is a closed disk of radius $p^{-1/2}$.
So $\CD(1,p^{-1/2}),g(\CD)\subset \SP(0,1)$ are the two disks of radius $p^{-1/2}$ which intersect $g(\P)$.
\end{proof}
For $x\in \mathbb{R}$, denote by $\lfloor x \rfloor $ the integeral part of $x$.
\begin{lemma}\label{ramify2}
For each $n\geq 1$, let $\CD\subset \SP(0,1)$ be a closed disk with radius $p^{-n/2}$. If $\CD\cap g(\P)\neq \emptyset$  and if we decompose $\CD$ into $p$
closed disk with radius $p^{-(n+1)/2}$, then there are $p^{\lfloor (n+1)/2\rfloor - \lfloor n/2\rfloor}$ such disks which intersect $g(\P)$.
\end{lemma}
\begin{proof}
By Proposition \ref{disk}
$$g(\mathbb{P}^{1}(K)\setminus \CD(x_{0},r_{0}))=\D(1,p^{-1/2}).$$
Observe that $r_{0}=p^{m/2}$ for some odd number $m\in \Z$. Applying Lemma \ref{rami-number}, we obtain the conclusion by similar arguments in the proof of  Lemma \ref{induction}.
\end{proof}

\begin{lemma}\label{ramify3}
For each $n\geq 1$, $\mathbb{S}(0,1)$ consists of $(p-1)p^{n-1}$ disjoint closed disks with radius $p^{-n/2}$. There are
$2 \cdot p^{\lfloor\frac{n}{2}\rfloor}$ such disks which intersect $g(\P)$.
\end{lemma}
\begin{proof}Using Lemmas \ref{ramify1} and \ref{ramify2}, we obtain the conclusion by induction.
\end{proof}

We are now ready to determine the minimal decomposition of the system $(\P,\phi)$ when $K$ is a ramified quadratic extension of $\Qp$.

\begin{proof}[Proof of Theorem \ref{rami-decomposition-p>2}]
 Observe that the dynamical system  $(\P,\phi)$ is conjugate to $(g(\P),\psi)$ by $g(x)=\frac{x-x_2}{x-x_1}$.
So we need only study the dynamics $(g(\P),\psi)$.

Since $K=\Qp(\sqrt{\Delta})$ is
a ramified quadratic extension of $\Qp$ and $\Delta\in \Qp$, it follows   that
$$|a+d|_{p}\neq |\sqrt{\Delta}|_{p}.$$
We distinguish the following two cases.\\
\indent \mbox{\rm (1)} \ Assume  $|a+d|_{p}> |\sqrt{\Delta}|_{p}$.  Write
$$\lambda=1+\frac{2\sqrt{\Delta}}{a+d-\sqrt{\Delta}}.$$
We have $\lambda=1 ~(\!\!\!\mod ~\pi)$.
By  Proposition \ref{affine} and Remark \ref{remark}, $(\SP(0,1),\psi)$ is decomposed into  $(p-1)p^{v_{\pi}(\lambda^{p}-1)-2}$ subsystems $(B_i,\psi)$ of type $(p,e)$ at level
$v_{\pi}(\lambda^{p}-1)$  and each $B_{i}$ is a union of $p$ closed disks of radius $p^{-v_{\pi}(\lambda^{p}-1)/2}$.
 By Lemmas \ref{ramify1} and \ref{ramify3},  there  are $2p^{(v_{\pi}(\lambda^p-1)-3)/2}$ such disks which intersect $g(\P)$.

If $B_{i}\cap g(\P)\neq \emptyset$, we claim that
$(B_{i}\cap g(\P),\psi)$ is minimal.

Let $\sigma=(x_{1},\cdots,x_p)$
be the cycle of $\psi_{v_{\pi}(\lambda^{p}-1)}$ at level $v_{\pi}(\lambda^{p}-1)$ with $x_{1}\in B_{i}$. So the cycle $\sigma $
 grows. Let $\sigma^{1}=(y_{1},\cdots,y_{p^{2}})$ be the lift of $\sigma$  at level $v_{\pi}(\lambda^{p}-1)+1$.  Then $\sigma^{1}$  splits one
time and all descendants at level $v_{\pi}(\lambda^{p}-1)+2$  grow.  By Lemma \ref{ramify2}, there is a descendant   $\sigma^{2}=(z_1,\cdots,z_{p^{2}})$ of $\sigma^{1}$ at level $v_{\pi}(\lambda^{p}-1)+2$ such that
$$\CD(z_i,p^{-(v_{\pi}(\lambda^{p}-1)+2)/2})\cap g(\P)\neq \emptyset, \mbox{~for~} i=1,\cdots,p^2 .$$
 Using Lemmas  \ref{minimal-part-to-whole} and \ref{ramify2}, by  induction, we infer that $(B_{i}\cap g(\P), \psi)$ is minimal. Following the above arguments, we also deduce that the minimal subsystem  $(B_{i}\cap g(\P), \psi)$
is conjugate to the adding machine on  the odometer $\mathbb{Z}_{(p_s)}$  with $(p_s)=(1,p,p^2,\cdots)$.

\indent \mbox{\rm (2)} Assume $|a+d|_{p}<|\sqrt{\Delta}|_{p}$. Write
$$\lambda=-1+\frac{2(a+d)}{a+d-\sqrt{\Delta}}.$$
Then $\lambda=-1 ~(\!\!\!\mod ~\pi)$ and $v_{\pi}(\lambda+1)=v_{\pi}(a+d)-v_{\pi}(\sqrt{\Delta}).$
So $\lambda^{2}=1~(\!\!\!\mod ~\pi)$ and $v_{\pi}(\lambda^{2p}-1)=v_{\pi}(\lambda^{p}+1)$.
We conclude by the same arguments as in Case (1).
\end{proof}

\subsection{Proof of Theorem \ref{decomp-p=2-unrami}}

Assume $K=\Q_{2}(\sqrt{-3})$.  Lemma \ref{unramifiedextension} shows that  $\Q_{2}(\sqrt{-3})$ is an unramified extension of $\Q_{2}$.
\begin{lemma}\label{p=2unrami}
If $\Q_{2}(\sqrt{\Delta})=\Q_{2}(\sqrt{-3})$, then the unit sphere $\SP(0,1)$ consists of $3$ closed disks of radius $1/2$ and each disk intersects $g(\P)$.
For $n\geq 1$, let $\CD\subset \SP(0,1)$ be a closed disk with radius $1/2^{n}$ and such that $\CD\cap \mathbb{P}^{1}(\Q_{2})\neq \emptyset$.
 Then $\CD$ consists of $4$  closed disks with radius $1/2^{n+1}$ and there are $2$
such disks which intersect $g(\mathbb{P}^{1}(\Q_{2}))$.
\end{lemma}
\begin{proof}
Observe that the unit sphere $\SP(0,1)$ consists of $3$  closed disks of radius $1/2$,
because  $K$ is an unramified extension of $\Q_{2}$.  By Proposition \ref{disk}, we have  $$g(\PK)\setminus \CD(x_{1},r_{0}))=\CD(1,1/2).$$
 By  Lemma \ref{distancep=2Lemma},
$$\CD(x_1,r_{0}/2)\cap \Q_{2}= \emptyset\quad \mbox{~and~}\quad \CD(x_2,r_{0}/2)\cap \Q_{2}= \emptyset.$$
 By Lemma \ref{unrami-number}, $\CD(x_{1},r_{0})$ consists of four closed disks of radius $r_{0}/2$ and there are two such disks which intersect $\Q_{2}$.  Let $\CD_{1},\CD_{2}\subset \CD(x_1,r_0)$ be the two closed disks of radius $r_{0}/2$ which intersect $\Q_{2}$.
By  Proposition \ref{disk},  $$g(\CD(x_1,r_{0}/2))=\mathbb{P}^{1}(K)\setminus \CD(0,1) \ \mbox{~and~} \ g(\CD(x_2,r_{0}/2))= \CD(0,1/2).$$
Observe that $$\CD(x_1,r_{0})=\CD(x_1,r_{0}/2)\cup\CD(x_2,r_{0}/2)\cup\CD_{1}\cup\CD_{2}.$$
So by Proposition \ref{disk}, $$\SP(0,1)= g(\CD_1)\cup g(\CD_2)\cup \CD(1,1/2)$$
where $g(\CD_1),g(\CD_2)$ are closed disks of radius $1/2$.

Analysis similar to that in the proof of Lemma \ref{induction} implies the remaining conclusion.
\end{proof}
 We are now ready to determine
the minimal decomposition of the dynamics $(\mathbb{P}^{1}(\Q_{2}),\phi)$  when $K=\Q_{2}(\sqrt{-3})$. Observe that  $2$ is a uniformizer of $\Q_2(\sqrt{-3})$.

\begin{proof}[Proof of Theorem \ref{decomp-p=2-unrami}]
Since the dynamics $(\mathbb{P}^{1}(\Q_{2}),\phi)$ is conjugate to $(g(\mathbb{P}^{1}(\Q_{2})),\psi)$, we are going to study the dynamics  $(g(\mathbb{P}^{1}(\Q_{2})),\psi)$. Observe that  $(g(\mathbb{P}^{1}(\Q_{2})),\psi)$ is a subsystem of $(\SP(0,1),\psi)$.
By Proposition \ref{affine}, we get $\ell|3$, i.e. $\ell=1 \mbox{ or } 3$. Since $\Q_{2}(\sqrt{-3})$ is an unramified quadratic extension of $\Q_{2}$, then $e=1$.
Proposition \ref{affine} and Remark \ref{remark} show that $(\SP(0,1),\psi)$ is decomposed into  $3\cdot 2^{v_{2}(\lambda^\ell-1)\cdot 2-2}/\ell$ subsystems $(B_i,\psi)$ of type $(\ell,\vec{E})$
at level $v_{2}(\lambda^\ell-1)$ with $\vec{E}=(v_{2}(\lambda^{\ell}+1),1,1,\cdots)$, and each $B_i$ is a $\psi$-invariant clopen set which is a union of $\ell$ closed disks of radius $2^{-v_{2}(\lambda^{\ell}-1)}$. By Lemma \ref{p=2unrami}, there are $3\cdot 2^{v_{2}(\lambda^{\ell}-1)-1}/\ell$ such clopen sets intersecting $g(\mathbb{P}^{1}(\Q_{2}))$.

Let   $(B,\psi)$ be a subsystem of type $(\ell,\vec{E})$
at level $v_{2}(\lambda^\ell-1)$ with $\vec{E}=(v_{2}(\lambda^{\ell}+1),1,1,\cdots)$ and such that $B\cap \mathbb{P}^{1}(\Q_{2}) \neq \emptyset$.
The system $(B,\psi)$ is decomposed into $2^{v_{2}(\lambda^{\ell}+1)2-2}$ subsystems $(C_{j},\psi)$ of type $(2\ell,1)$ at level $v_{2}(\lambda^{2\ell}-1)$ and there are $2^{v_{2}(\lambda^{\ell}+1)-1}$ such subsystems  of $(B,\psi)$ which intersect $g(\Q_p)$.

If $C_j\cap g(\mathbb{P}^{1}(\Q_{2}))\neq \emptyset$, we claim that the system $(C_j\cap g(\mathbb{P}^{1}(\Q_{2})),\psi)$ is minimal.
 Observe that  $(C_j,\psi)$ is of type $(2\ell,1)$ at $v_{2}(\lambda^{2\ell}-1)$. Let $(x_{1},\cdots,x_{2\ell})$ be the cycle of $\psi_{v_{2}(\lambda^{2\ell}-1)}$
at level $v_{2}(\lambda^{2\ell}-1)$ corresponding to $C_i$. Then  the cycle $(x_{1},\cdots,x_{2\ell})$  grows. By Lemma \ref{p=2unrami}, there exists a unique lift $(y_1,\cdots,y_{4\ell })$ at level $v_{2}(\lambda^{2\ell}-1)+1$ such that
$$\CD(y_k,p^{-v_{2}(\lambda^{2\ell}-1)-1})\cap g(\mathbb{P}^{1}(\Q_{2}))\neq  \emptyset \quad\mbox{for~} 1\leq k\leq 4\ell.$$
Using Lemma \ref{minimal-part-to-whole}, by induction, we conclude that the system $(C_{j}\cap g(\mathbb{P}^{1}(\Q_{2})),\psi)$ is minimal.

So the dynamical system $(\mathbb{P}^{1}(\Q_2),\phi)$ is decomposed into  $3\cdot2^{v_{2}(\lambda^{2\ell}-1)-2 }/\ell$ minimal subsystems.
Following the above arguments, we also know that each minimal subsystem  is conjugate to  the
adding machine on the odometer $\mathbb{Z}_{(p_s)}$  with $(p_s)=(\ell,\ell 2,\ell 2^2,\cdots)$.
\end{proof}

\subsection{Proof of Theorem \ref{decompositionsqrt2}}
Assume that  $K=\Q_2(\sqrt{i}),$ where $i=2,-2,6 \mbox{~or~} -6 $. By Lemma \ref{unramifiedextension}, $K$ is a ramified quadratic extension of $\Q_{2}$. So we have $ \SP(0,1)=\CD(1,\sqrt{2}/2)$.
\begin{lemma}\label{ramisqrt2}
We have \\
\indent {\rm (1)} $\CD(1,\sqrt{2}/2)=\CD(1,1/2)\cup\CD(1+\sqrt{i},1/2)$, $ \CD(1,1/2)\cap g(\mathbb{P}^{1}(\Q_{2}))\neq \emptyset$
 and $\CD(1+\sqrt{i},1/2)\cap g(\mathbb{P}^{1}(\Q_{2}))= \emptyset.$ \\
\indent {\rm (2)} $\CD(1,1/2)$ consists of $2$ closed disks of radius $\sqrt{2}/4$, all of  which intersect  $g(\mathbb{P}^{1}(\Q_{2}))$.
\\
\indent {\rm (3)} For $n\geq 3$, let $\CD\subset \CD(1,1/2) $ be a closed disk of  radius $2^{-n/2}$ with $$\CD \cap g(\mathbb{P}^{1}(\Q_{2})) \neq \emptyset.$$ Then $\CD$ consists of $2$ closed
disks of radius $2^{-(n+1)/2}$, and there are $2^{\lfloor (n+1)/2\rfloor-\lfloor n/2\rfloor}$ such disks intersecting $g(\mathbb{P}^{1}(\Q_{2}))$.
\end{lemma}
\begin{proof}
\indent {\rm (1)} By  Proposition \ref{disk}, we have $g(\PK\setminus \CD(x_{1},\sqrt{2}r_{0}))=\CD(1,1/2).$
Lemma \ref{distancep=2Lemma} shows that $\CD(x_1,\sqrt{2}r_0)\cap \Q_{2}=\emptyset$. Observe that $\CD(1,\sqrt{2}/2)=\CD(1,1/2)\cup \CD(1+\sqrt{i},1/2)$ and
$\mathbb{P}^{1}(\Q_{2})\subset \PK\setminus \CD(x_{1},\sqrt{2}r_{0})$. So
 $$\CD(1+\sqrt{i},1/2)\cap g(\mathbb{P}^{1}(\Q_{2}))=\emptyset  \mbox{ and } g(\mathbb{P}^{1}(\Q_{2}))\subset \CD(1,1/2).$$

\indent {\rm (2)} By Proposition \ref{disk},
$$g(\PK\setminus \CD(x_{1},2r_{0}))=\CD(1,\sqrt{2}/4) \mbox{~and~} g(\SP(x_1,2r_{0}))=\SP(1,1/2) .$$
By Lemma \ref{distancep=2Lemma}, $\SP(x_1,2r_{0})\cap\Q_{2}\neq \emptyset$. Observe that $\CD(1,1/2)=\CD(1,\sqrt{2}/4)\cup \SP(1,1/2)$,
where $\SP(1,1/2)$ is actually a closed disk of radius $\sqrt{2}/4$. So we can conclude.\\
\indent {\rm (3)}
Observe that $r_0=p^{m/2}$ for some odd integer $m$.  Using Proposition \ref{disk}, Lemma \ref{rami-number} and the analysis similar to that in the proof of  Lemma \ref{induction}, we obtain the conclusion.
\end{proof}

\begin{proof}[Proof of Theorem \ref{decompositionsqrt2}]
 We need only study the dynamical system $(g(\mathbb{P}^{1}(\Q_{2})),\psi)$, since $(\mathbb{P}^{1}(\Q_{2}),\phi)$ is conjugate to $(g(\mathbb{P}^{1}(\Q_{2})),\psi)$ by $g=\frac{x-x_2}{x-x_1}$. It is easy to see that $|a+d|_{2}\neq |\sqrt{\Delta}|_{2}$. Thus we distinguish two cases.\\
\indent {\rm (1)} Assume  $|a+d|_{2}> |\sqrt{\Delta}|_{2}$.  Write
$$\lambda=1+\frac{2\sqrt{\Delta}}{a+d-\sqrt{\Delta}}.$$
We have $v_{\pi}(\lambda-1)=2+v_{\pi}(\sqrt{\Delta})-v_{\pi}(a+d)\geq 3$ and $v_{\pi}(\lambda+1)=2$.
By Proposition \ref{affine} and Remark \ref{remark}, $(\SP(0,1),\psi)$ is decomposed into  $2^{v_{\pi}(\lambda-1)-1}$ subsystems $(B_i,\psi)$ of type $(1,2)$
at level $v_{\pi}(\lambda-1)$, each $B_i$ is a $\psi$-invariant clopen set which is a closed disk of radius $2^{-v_{\pi}(\lambda-1)}$.
Observe that $v_{\pi}(\lambda-1)$ is odd. By Lemma \ref{ramisqrt2}, there are $2^{(v_{\pi}(\lambda-1)-1)/2}$ such clopen sets of  type $(\ell,2)$ intersecting $g(\mathbb{P}^{1}(\Q_{2}))$.

If $B_{i}\cap g(\mathbb{P}^{1}(\Q_{2}))\neq \emptyset$, analysis similar to that in the proof of Theorem \ref{rami-decomposition-p>2} implies that the dynamical system
$(B_{i}\cap g(\mathbb{P}^{1}(\Q_{2})),\psi)$ is minimal.

So the dynamical system $(\mathbb{P}^{1}(\Q_{2}),\phi)$ is decomposed into $2^{(v_{\pi}(\lambda-1)-1)/2}$ minimal subsystems and each minimal system
is conjugate to  the
adding machine on the odometer $\mathbb{Z}_{(p_s)}$  with $(p_s)=(1, 2, 2^2,\cdots)$.\\
\indent {\rm (2)}
Assume  $|a+d|_{2}< |\sqrt{\Delta}|_{2}$.  Write
$$\lambda=-1+\frac{2(a+d)}{a+d-\sqrt{\Delta}}.$$
We get  $v_{\pi}(\lambda+1)=2+v_{\pi}(\sqrt{\Delta})-v_{\pi}(a+d)\geq 3$ and $v_{\pi}(\lambda-1)=2$.

By Proposition \ref{affine} and Remark \ref{remark}, $(\SP(0,1),\psi)$ is decomposed into  $2$ subsystems $(B_i,\psi)$ of type $(1,\vec{E})$ at level $2$ with
 $\vec{E}=(v_{2}(\lambda+1),2,2,\cdots)$, and each $B_{i}$ is  a $\psi$-invariant closed disk of radius $2^{-1}$, where $i=1\mbox{~or~} 2$.
 By Lemma \ref{ramisqrt2}, there is a unique such closed disk which intersects $g(\mathbb{P}^{1}(\Q_{2}))$. Without loss of generality, we may suppose that $ B_{1}\cap g(\mathbb{P}^{1}(\Q_{2}))\neq \emptyset$.

Being of type $(1,\vec{E})$ at level $2$, the dynamics $(B_1,\psi)$ is decomposed into $2^{(v_{\pi}(\lambda+1)-1)}$ subsystems
 of type $(2,2)$ at level $v_{\pi}(\lambda^{2}-1)$. By Lemma   \ref{ramisqrt2}, there are $2^{(v_{\pi}(\lambda+1)-1)/2}$ such clopen sets of
 type $(\ell,2)$ intersecting $g(\mathbb{P}^{1}(\Q_{2}))$.  If $(C,\psi)$ is such a subsystem satisfying $C\cap g(\mathbb{P}^{1}(\Q_{2}))\neq \emptyset$,
analysis similar to that in the proof of Theorem \ref{rami-decomposition-p>2} implies that the dynamical system $(C\cap g(\mathbb{P}^{1}(\Q_{2})),\psi)$ is minimal.

So the dynamical system $(\mathbb{P}^{1}(\Q_{2}),\phi)$ is decomposed into $2^{(v_{\pi}(\lambda+1)-1)/2}$ minimal subsystems and each minimal system
is conjugate to  the
adding machine on the odometer $\mathbb{Z}_{(p_s)}$  with $(p_s)=(1, 2, 2^2,\cdots)$.\\
\end{proof}

\subsection{Proof of Theorem \ref{decompositionsqrt3}}
Assume that $K=\Q_2(\sqrt{-1})$ or $\Q_2(\sqrt{3})$.  By Lemma \ref{unramifiedextension}, $K$ is a ramified quadratic extension of $\Q_{2}$. We also have
$\SP(0,1)=\CD(1,\sqrt{2}/2)$.
\begin{lemma}\label{numbersqrt3or-1}We have \\
\indent {\rm (1)} the disk $\CD(1,\sqrt{2}/2)$ consists of $2$ closed disks of radius $1/2$, all of which  intersect $g(\mathbb{P}^{1}(\Q_{2}))$;\\
\indent {\rm (2)} for $n\geq 2$, let $\CD\subset \CD(1,\sqrt{2}/2) $ be a closed disk of  radius $2^{-n/2}$. Then $\CD$ consists of $2$ closed disks of radius $2^{-(n+1)/2}$ and there are $2^{\lfloor n/2\rfloor-\lfloor (n-1)/2\rfloor}$ such disks which intersect $g(\mathbb{P}^{1}(\Q_{2}))$.
\end{lemma}
\begin{proof}
\indent {\rm (1)}
 By Proposition \ref{disk}, $$g(\PK\setminus \CD(x_{1},r_{0}))=\CD(1,\sqrt{2}/2).$$
 Observe that $\CD(1,\sqrt{2}/2)=\CD(1,1/2)\cup \SP(1,\sqrt{2}/2)$ and $\SP(1,\sqrt{2}/2)$ is a closed disk with radius $1/2$. By Lemma
\ref{distancep=2Lemma} and Proposition \ref{disk}, $\SP(1,\sqrt{2}/2)\cap g(\P)\neq \emptyset$. Clearly, $\CD(1,1/2)\cap g(\P)\neq \emptyset$, since $g(\infty)=1$.\\
\indent {\rm (2)}
 We conclude by arguments similar to that in the proof of Lemma \ref{ramisqrt2}.
\end{proof}

\begin{lemma}\label{distancea+d}
Assume $\Q_2(\sqrt{\Delta})=\Q_2(\sqrt{i})$ where $i=-1$ or $3$. If $|a+d|_{2}=|\sqrt{\Delta}|_{2}$, then $$|a+d-\sqrt{\Delta}|_{2}=\frac{\sqrt{2}}{2}\cdot |a+d|_{2}.$$
\end{lemma}
\begin{proof} It is a immediate consequence of  the second assertion of Proposition \ref{x-sqrt}.
\end{proof}

\begin{proof}[Proof of Theorem \ref{decompositionsqrt3}]
 We are going to study the dynamical system  $(g(\mathbb{P}^{1}(\Q_{2})),\psi)$, since $(\mathbb{P}^{1}(\Q_{2}),\phi)$ is conjugate to $(g(\mathbb{P}^{1}(\Q_{2})),\psi)$ by $g=\frac{x-x_2}{x-x_1}$.  We distinguish the following three cases.\\
\indent {\rm (1)}
 Assume  $|a+d|_{2}= |\sqrt{\Delta}|_{2}$.  Write
$$\lambda=1+\frac{2\sqrt{\Delta}}{a+d-\sqrt{\Delta}}.$$
By Lemma \ref{distancea+d}, $v_{\pi}(\lambda-1)=2+v_{\pi}(\sqrt{\Delta})-v_{\pi}(a+d-\sqrt{\Delta})=1$ and $v_{\pi}(\lambda+1)=1$.
For $n\geq 2$, it is easy to check $v_{\pi}(\lambda^{2^{n}}+1)=2$.
By Proposition \ref{affine}, $(\SP(0,1),\psi)$ is  of type $(2,\vec{E})$ with $\vec{E}=(v_{\pi}(\lambda^{2}+1),2,2,\cdots)$
at level $2$.

Let $\sigma=(x_{1},x_{2})$
be the cycle  at level $2$. So the cycle $\sigma$ will grow.  Let $\sigma^{1}=(y_{1},\cdots,y_{4})$ be the lift of $\sigma$  at level $3$.
 Then $\sigma^{1}$  splits $v_{\pi}(\lambda^{2}+1)-1$ times and  all the descendants of $\sigma^{1}$  at level $v_{\pi}(\lambda^{4}-1)$ grow.
By a simple calculation, $$\lambda^{2}+1=\frac{2(a+d)^{2}+2\Delta}{(a+d-\sqrt{\Delta})^{2}}.$$
Lemma \ref{distancea+d} leads to $v_{\pi}(a+d-\sqrt{\Delta})=1+ v_{\pi}(a+d)$. Since $v_{\pi}(2)=2$ and
 $v_{\pi}((a+d)^{2}+\Delta)\geq v_{\pi}(a+d)^{2}+2 $ are even, $v_{\pi}(\lambda^{2}+1)\geq 2$ is also an even number.

By Proposition \ref{affine}, the dynamics $(\SP(0,1),\psi)$ is decomposed into $2^{{v_{\pi}(\lambda^{2}+1)-1}}$ subsystems $(B_i, \psi)$
of type $(4,e)$ at level $v_{\pi}(\lambda^{2}+1)+2$.  By Lemma \ref{numbersqrt3or-1}, there are $2^{({v_{\pi}(\lambda^{2}+1)-2})/2}$ such components of type
$(4,e)$ which intersect $g(\mathbb{P}^{1}(\Q_{2})$.
If $B_{i}\cap g(\mathbb{P}^{1}(\Q_{2}))\neq \emptyset$, analysis similar to that in the proof of Theorem \ref{rami-decomposition-p>2} implies that the dynamical system
$(B_{i}\cap g(\mathbb{P}^{1}(\Q_{2})),\psi)$ is minimal.

So the dynamical system $(\mathbb{P}^{1}(\Q_{2}),\phi)$ is decomposed into $2^{(v_{\pi}(\lambda^{2}+1)-2)/2}$ minimal subsystems and each minimal system
is conjugate to  the
adding machine on the odometer $\mathbb{Z}_{(p_s)}$  with $(p_s)=(1, 2, 2^2,\cdots)$.\\
\indent {\rm (2)}
 Assume  $|a+d|_{2}> |\sqrt{\Delta}|_{2}$.  Write
$$\lambda=1+\frac{2\sqrt{\Delta}}{a+d-\sqrt{\Delta}}.$$
We obtain $v_{\pi}(\lambda-1)=2+v_{\pi}(\sqrt{\Delta})-v_{\pi}(a+d)\geq 4$ and $v_{\pi}(\lambda+1)=2$.
 We conclude by arguments similar to that in the proof of Theorem \ref{decompositionsqrt2}.\\
\indent {\rm (3)}
 Assume  $|a+d|_{2}< |\sqrt{\Delta}|_{2}$.  Write
$$\lambda=-1+\frac{2(a+d)}{a+d-\sqrt{\Delta}}.$$
We have $v_{\pi}(\lambda+1)=2+(v_{\pi}(a+d))-v_{\pi}\sqrt{\Delta}\geq 4$ and $v_{\pi}(\lambda-1)=2$.
We also conclude by arguments similar to that in the proof of Theorem \ref{decompositionsqrt2}.
\end{proof}


\section{Invariant measure}\label{inv}
As we have seen when $\phi$ admits no fixed point in $\Qp$ and $\phi^{n}\neq  id$ for all positive integer $n$,
the dynamics $(\P,\phi)$ is decomposed into a finite number of minimal subsystems. In this section, we determine the invariant probability measures for each minimal subsystem.

Let $(X_{1},d_{1})$ and $(X_{2},d_{2})$ be metric spaces, and let $\mathcal{F}$ be a family of maps from $X_{1}$ to $X_{2}$.
The collection $\mathcal{F}$ is said to be \emph{equicontinuous} on $X_1$, if for every $x\in X_1$ and  every $\epsilon >0$, there exists a $\delta>0$  such that
$$d_{1}(x,y)<\delta\Longrightarrow  d_{2}(f(x),f(y)) \mbox{~for every $f\in \mathcal{F}$}.$$

For each $g\in {\rm{PGL}}(2,\Q_p)$, there exist a  number $C>0$ such that
$$\rho(g(x), g(y))\leq \rho(x,y)\mbox{~for all~ } x,y \in \Qp,$$ (see Theorem 2.14 of \cite{Silverman-dynamics-book}).

Since the dynamics $(\P,\phi)$ is
conjugate to $(g(\P),\lambda x)$ with $|\lambda|_{p}=1$ by a conjugacy $g\in {\rm{PGL}}(2,\Q_p)$ defined by $g(x)=\frac{x-x_2}{x-x_1}$, the dynamics $\phi: \P\rightarrow \P$ is equicontinuous 
with respect to the chordal metric, where the dynamics is equicontinuous means the family of iterates
$\{\phi^{i}\}_{i\geq 1}$ is
equicontinuous on $\P$.
 Theorem \ref{uniquelyergodic} below shows that for each minimal subsystem there exists a unique invariant (ergodic) measure with the minimal component as its support.


\begin{theorem}[\cite{Oxtoby}]\label{uniquelyergodic}
Let $X$ be a compact metric space and $T:X\mapsto X$ be an equicontinuous transformation. Then the following statements are
equivalent:\\
\noindent{\rm (1)} $T$ is minimal on $X$.\\
\noindent{\rm (2)} $T$ is uniquely ergodic on $X$.\\
\noindent{\rm (3)} $T$ is ergodic for any ( or some) invariant probability measure with $X$ as its support.
\end{theorem}

Let $\mu_{0}$ be the Haar measure on $\Zp$. We define two probability measures  $\hat{\mu}$ and $\overline{\mu}$ on the
projective line $\P$. For each Borel measurable set $S\subset\P$, let \[S_{1}=\Zp\cap S, \qquad S_{2}=(\P\setminus\Zp)\cap S.\]
Define $\xi: \P\rightarrow \P $ by $\xi(z)=1/z$.
 It is clear that $\xi^{-1}(S_{2})\subset \Zp$.  The image of
$\mathbb{P}^1(\mathbb{Q}_p)\setminus \mathbb{Z}_p$ under $\xi$
is the ball $p\Z_p$. So
$\mu_0 \circ \xi^{-1} (\P\setminus \Z_p) =1/p.$ The measure  $\hat{\mu}$  is then defined by
$$\hat{\mu}(S):=\frac{p}{p+1} \big(\mu_{0}(S_1)+ \mu_{0}\circ \xi^{-1}(S_2)\big).$$
The measure  $\overline{\mu}$ is defined by
 $$\overline{\mu}(S):=\frac{1}{2}\mu_{0}(S_1)+ \frac{p}{2} \cdot \mu_{0} \circ \xi^{-1}(S_2).$$

First, we discuss the case when $p\geq 3$.

Assume that $\Qp(\sqrt{\Delta})$ is an unramified quadratic extension of $\Qp$. Since $p\geq 3$, it follows that  $|x_{1}-\frac{a-d}{2c}|_{p}=|\frac{\sqrt{\Delta}}{2c}|_p=|\frac{\sqrt{\Delta}}{c}|_p=r_0$.  Then $\CD(x_{1},r_{0})=\CD(\frac{a-d}{2c},r_{0})$.
So $\CD(x_{1},r_{0})\cap \Qp=\CD(\frac{a-d}{2c},r_{0})\cap \Qp= \overline{D}(\frac{a-d}{2c},r_{0})$. Let $\eta \in\Qp$ be a point with $|\eta|_{p}=r_{0}$ and set $$h(x)=\eta x+\frac{a-d}{2c}.$$
Then $h^{-1}(\overline{D}(\frac{a-d}{2c},r_{0}))=\overline{D}(0,1)$. According to the proof of Theorem \ref{mindecunrami}, it is easy to prove the following theorem.
\begin{theorem}\label{measure1}
Under the  assumptions of Theorem \ref{mindecunrami}, if  $(B,\phi)$ is a minimal subsystem of $(\P,\phi)$, then the unique invariant measure $\hat{\sigma}$
 of the system $(B,\phi)$ is defined as follows: for each measurable subset $A\subset B$,
$$\hat{\sigma}(A)=\frac{\hat{\mu}\circ h^{-1}(A)}{\hat{\mu}\circ h^{-1}(B)}.$$
\end{theorem}

\begin{proof}
Let $\varphi=h^{-1}\circ \phi \circ h$. Then $h^{-1}(x_1)$ and $h^{-1}(x_2)$ are the two fixed points of $\varphi$.
It is easy to check that $|h^{-1}(x_1)|_{p}=|h^{-1}(x_2)|_{p}=|h^{-1}(x_1)-h^{-1}(x_2)|_p=1$.

We assert that the measure $\hat{\mu}$ is $\varphi$-invariant. Since $(B,\phi)$ is a minimal, it follows that $(h^{-1}(B),\varphi)$ is minimal.  So  the measure $\hat{\mu}|_{h^{-1}(B)}$ (the restriction of measure $\hat{\mu}$ on $h^{-1}(B)$) is an invariant measure of dynamical system $(h^{-1}(B),\varphi)$. Thus the measure $\hat{\sigma}$ is an invariant measure
 of the system $(B,\phi)$.

Now we are going to prove our assertion. Actually, if the two fixed points $x_{1}, x_{2}$ of $(\P,\phi)$ satisfy that $|x_1|_p=|x_2|_p=|x_1-x_2|_p=1$, then $\hat{\mu}$ is $\phi$-invariant.  Since the dynamical system $(\P,\phi)$ is conjugate to the system $(g(\P),\psi)$ with $\psi(x)=\lambda x$ by $g(x)=\frac{x-x_2}{x-x_1}$.
For  a positive integer $n\geq 1$ and  $a\in g(\P) $, let $$\widetilde{D}(a,p^{n}):=\{x\in g(\P):|x-a|_{p}\leq r\}$$ be a disk of radius $p^{n}$
centred at $a$ in $g(\P)$. For each $\widetilde{D}(a,p^{n})\subset g(\P)$,  we define function
$\widetilde{\mu}(\widetilde{D}(a,p^{n}))=(p+1)^{-1}p^{-n+1}$. By Lemmas \ref{unramify1} and \ref{induction},  the function $\widetilde{\mu}$ can be extended to be a Borel probability measure on $g(\P)$. Since $\psi(x)=\lambda x$ with $|\lambda|_{p}=1$, it follows that
the measure $\widetilde{\mu}$ is $\psi$-invariant.
Also it is easy to check that $\hat{\mu}=\widetilde{\mu}\circ g^{-1}$.
So $\hat{\mu}$ is $\phi$-invariant, which completes the proof.
\end{proof}

Now assume that $\Qp(\sqrt{\Delta})$ is an ramified quadratic extension of $\Qp$. It is easy to see that  $\CD(x_{1},r_{0})\cap \Qp= \overline{D}(\frac{a-d}{2c},p^{-1/2}\cdot r_{0})$. Let $\eta \in\Qp$ be a point with $|\eta|_{p}=p^{-1/2}\cdot r_{0}$ and set $$h(x)=\eta x+\frac{a-d}{2c}.$$
Then $h^{-1}(\overline{D}(\frac{a-d}{2c},p^{-1/2}\cdot r_{0}))=\overline{D}(0,1)$. According to the proof of Theorem \ref{rami-decomposition-p>2}, we can easily prove the following theorem.
\begin{theorem}Under the  assumptions of Theorem \ref{rami-decomposition-p>2}, if $(B,\phi)$ is a minimal subsystem of $(\P,\phi)$, the unique invariant measure $\overline{\sigma}$  of the system $(B,\phi)$ is defined as follows: for each measurable subset $A\subset B$,
$$\overline{\sigma}(A)=\frac{\overline{\mu}\circ h^{-1}(A)}{\overline{\mu}\circ h^{-1}(B)}.$$
\end{theorem}
\begin{example}\label{example1}
Let $p=3$ and consider the dynamical system $(\mathbb{P}^{1}(\Q_{3}),\phi)$ where $$\phi(x)=\frac{1}{x+1}.$$
\end{example}
Consider the equation $\phi(x)=x$, where discriminant
$\Delta=5$.
 By Lemma \ref{solution}, $\sqrt{5}\notin \Q_{3}$. Then $\phi$ has two fixed points $x_{1}=\frac{-1+\sqrt{5}}{2}$
and $x_{2}=\frac{-1-\sqrt{5}}{2}$ in  $K=\Q_{3}(\sqrt{5})$ which is an unramified quadratic extension of $\Q_{3}$.  On the other hand,
$$\lambda=\frac{1+\sqrt{5}}{1-\sqrt{5}}=-\frac{3+\sqrt{5}}{2}.$$
Then $$\lambda^{2}=\frac{7+3\sqrt{5}}{2}=-1 \ (\!\!\!\!\!\mod 3)$$ and
$$\lambda^{4}=\frac{47+21\sqrt{5}}{2}=1 \ (\!\!\!\!\!\mod 3), \quad v_{3}(\lambda^{4}-1)=1.$$
By Theorem \ref{mindecunrami},  the system $(\mathbb{P}^{1}(\Q_{3}),\phi)$ is minimal and conjugate to
the adding machine on the odometer $\mathbb{Z}_{(p_{s})}$, where
$$(p_{s})=(4,4\cdot3,4\cdot3^{2},\cdots).$$

By Theorem \ref{uniquelyergodic}, there exists a unique invariant probability measure with $\mathbb{P}^{1}(\Q_{3})$ as its support.
Notice that   $r_{0}=|x_{1}-x_{2}|_{3}=|\sqrt{5}|_{3}=1$ and $\overline{D}(-\frac{1}{2},1)=\overline{D}(0,1)$. By
Theorem \ref{measure1}, the invariant measure of system $(\mathbb{P}^{1}(\Q_{3}),\phi)$ is $\hat{\mu}$.

Now let us consider the case  $p=2$.

Assume that $K=\Q_{2}(\sqrt{-3})$. Notice that $2$ is a uniformizer of $\Q_{2}(\sqrt{\Delta})$. It is easy to see that
 $\CD(x_{1},r_{0})\cap \Q_{2}= \overline{D}(\frac{a-d- p^{v_{2}(\sqrt{\Delta})}}{2c},r_{0})$. Let $\eta \in\Q_{2}$ be a point with $|\eta|_{2}=r_{0}$ and set $$h(x)=\eta x+\frac{a-d- p^{v_{2}(\sqrt{\Delta})}}{2c}.$$
Then $h^{-1}(\overline{D}(\frac{a-d- p^{v_{2}(\sqrt{\Delta})}}{2c},r_{0}))=\overline{D}(0,1)$. According to the proof of Theorem \ref{decomp-p=2-unrami}, we have the following theorem.
\begin{theorem}\label{measure1}
Under the  assumptions of Theorem \ref{decomp-p=2-unrami}, if  $(B,\phi)$ is a minimal subsystem of $(\mathbb{P}^{1}(\Q_{2}),\phi)$, the unique invariant measure $\hat{\sigma}$ of the system $(B,\phi)$ is defined as follow:  for each measurable subset $A\subset B$,
$$\hat{\sigma}(A)=\frac{\hat{\mu}\circ h^{-1}(A)}{\hat{\mu}\circ h^{-1}(B)}.$$
\end{theorem}

Assume that $\Q_{2}(\sqrt{\Delta})=\Q_{2}(\sqrt{i})$ where $i=2,-2,6 \mbox{ or } -6$. It is easy to see that  $\CD(x_{1},2r_{0})\cap \Q_{2}= \overline{D}(\frac{a-d}{2c},\sqrt{2}\cdot r_{0})$. Let $\eta \in\Q_{2}$ be a point with $|\eta|_{2}=\sqrt{2}\cdot r_{0}$ and set $$h(x)=\eta x+ \frac{a-d}{2c}.$$
Then $h^{-1}(\overline{D}(\frac{a-d}{2c},\sqrt{2}\cdot r_{0}))=\overline{D}(0,1)$. According to the proof of Theorem \ref{decompositionsqrt2}, we can obtain the following theorem.
\begin{theorem}Under the  assumptions of Theorem \ref{decompositionsqrt2}, if  $(B,\phi)$ is a minimal subsystem of $(\mathbb{P}^{1}(\Q_{2}),\phi)$, the unique invariant measure $\overline{\sigma}$  of the system $(B,\phi)$ is defined as follow: for each measurable subset $A\subset B$,
$$\overline{\sigma}(A)=\frac{\overline{\mu}(h^{-1}(A))}{\overline{\mu}(h^{-1}(B))}.$$
\end{theorem}

Assume that $\Q_{2}(\sqrt{\Delta})=\Q_{2}(\sqrt{i})$ where $i=-1 \mbox{ or } 3$ .  Let $\pi$ be a uniformizer of $\Q_{2}(\sqrt{\Delta})$.
It is easy to see that  $\CD(x_{1},\sqrt{2}\cdot r_{0})\cap \Q_{2}= \overline{D}(\frac{a-d- p^{v_{\pi}(\sqrt{\Delta})}}{2c},r_{0})$.
 Let $\eta \in\Q_{2}$ be a point with $|\eta|_{2}= r_{0}$ and set $$h(x)=\eta x+\frac{a-d- p^{v_{\pi}(\sqrt{\Delta})}}{2c}.$$
Then $h^{-1}(\overline{D}(\frac{a-d- p^{v_{\pi}(\sqrt{\Delta})}}{2c}, r_{0}))=\overline{D}(0,1)$. According to the proof of Theorem \ref{decompositionsqrt3}, we have the following theorem.
\begin{theorem}Under the assumptions of Theorem \ref{decompositionsqrt3}, if  $(B,\phi)$ is a minimal subsystem of $(\mathbb{P}^{1}(\Q_{2}),\phi)$, the unique invariant measure  $\overline{\sigma}$  of the system $(B,\phi)$ is defined as follow: for each measurable subset $A\subset B$,
$$\overline{\sigma}(A)=\frac{\overline{\mu}\circ h^{-1}(A)}{\overline{\mu}\circ h^{-1}(B)}.$$
\end{theorem}

\begin{example}
Let $p=2$ and consider the dynamical system $(\mathbb{P}^{1}(\Q_{2}),\phi)$ where $$\phi(x)=\frac{1}{x+1}.$$
\end{example}
Consider the equation $\phi(x)=x$, whose discriminant is
$\Delta=5$. By Lemma \ref{solution}, we have $\sqrt{5}\notin \Q_{2}$. Notice that $\sqrt{\frac{-3}{5}}\in \Q_{2}$. This implies that $\Q_{2}(\sqrt{5})=\Q_{2}(\sqrt{-3})$. Similar to Example \ref{example1}, $\phi$ has two fixed points $x_{1}=\frac{-1+\sqrt{5}}{2}$ and $x_{2}=\frac{-1-\sqrt{5}}{2}$ in  $K=\Q_{2}(\sqrt{5})$ which is an unramified quadratic extension of $\Q_{2}$.  So we have
$$\lambda=\frac{1+\sqrt{5}}{1-\sqrt{5}}=-\frac{3+\sqrt{5}}{2}$$
and $\lambda\not\equiv1 (\!\!\!\mod 2)$.
Thus $$\lambda^{3}=-9-4\sqrt{5}=1 (\!\!\!\!\mod 2),$$
$$\lambda^{6}=161+72\sqrt{5}, \quad  v_{2}(\lambda^{6}-1)=3.$$
Theorem \ref{decomp-p=2-unrami} now shows that $(\mathbb{P}^{1}(\Q_{2}),\phi)$ is decomposed into $2$ minimal subsystems which are conjugate to the adding machine on the odometer $\mathbb{Z}_{(p_{s})}$ with
$$(p_{s})=(3,3\cdot2,3\cdot2^{2},\cdots).$$
It is easy to check that
 $B_1=\overline{D}(0,1/8)\cup \overline{D}(1,1/8)\cup\overline{D}(1/2,1)\cup\overline{D}(2/3,1/8)\cup\overline{D}(3/5,1/8)\cup(\P\setminus \overline{D}(3/5,4))$
and $B_{i}=\overline{D}(2,1/8)\cup \overline{D}(1/3,1/8)\cup\overline{D}(3/4,2)\cup\overline{D}(4/7,1/8)\cup\overline{D}(1/11,1/8)\cup\overline{D}(18/11,1)$
are the two minimal components.

Notice that   $r_{0}=|x_{1}-x_{2}|_{3}=|\sqrt{5}|_{2}=1$ and $\CD(x_1,1)\cap\Q_{2} =\overline{D}(0,1)$. For $i=1 \mbox{~or~} 2$,
the invariant measure for $(B_i,\phi)$ is $\hat\mu|_{B_{i}}$, the restriction of $\hat{\mu}$ on $B_i$.

\section*{Acknowledgement}
The authors gratefully acknowledge the support of MCM, CAS, especially the support to the Program of Stochastics, Dimension and Dynamics.
The CNRS program (PICS No.5727) and  NSF of China (Grant No. 10831008 and 11231009) are also acknowledged.
\section*{Notation}
$$\begin{array}{ll}
    p & \hbox{a prime number;} \\
    \Zp & \hbox{ring of $p$-adic intgers;} \\
    \Qp & \hbox{field of $p$-adic numbers;}\\
     D(a,r)& =\{x\in \Qp:|x-a|< r\};\\
     \overline{D}(a,r)& =\{x\in \Qp:|x-a|\leq r\};\\
     U& =\{x\in \Qp: |x|_p=1\};\\
     V&=\{x\in \Qp: \exists~ n\in \N^{+},  x^{n}=1 \};\\
     K& \hbox{a finite extension of $\Qp$;} \\
     \mathcal{O}_{K}&=\{x\in K : |x|_p\leq 1\};\\
     \pi & \hbox{a  uniformizer of $K$;}\\
     e & \hbox{the ramification index of $K$ over $\Qp$;}\\
     |\cdot|_{p} & \hbox{the abosulute value on $\Qp$ or $K$;}\\
     \mathbb{D}(a,r)& =\{x\in K:|x-a|< r\};\\
     \overline{\mathbb{D}}(a,r)& =\{x\in K:|x-a|\leq r\};\\
          \mathbb{U}&=\{x\in\mathcal{O}_{K}:|x|_{p}=1\} ;\\
     \mathbb{V}&=\{x\in
\mathbb{U}:~\exists~ m\in \mathbb{N}^*, x^{m}=1\};\\
 \mathbb{S}(a,r)&=\{x\in K : |x-a|_{p}= r\};\\
\end{array}$$

\end{document}